\documentclass[a4paper,10pt]{amsart}
\usepackage[matrix,arrow,tips,curve]{xy}
\usepackage[english]{babel}
\usepackage[leqno]{amsmath}
\usepackage{amssymb,amsthm}
\usepackage{amscd}
\usepackage{enumerate}

\newtheorem{thm}{Theorem}[subsection]
\newtheorem{lemma}[thm]{Lemma}
\newtheorem{lemmadefi}[thm]{Lemma - Definition}

\newtheorem{prop}[thm]{Proposition}

\newtheorem{cor}[thm]{Corollary}

\newtheorem{fact}[thm]{Fact}

\theoremstyle{remark}
\newtheorem{remark}[thm]{Remark}
\theoremstyle{definition}
\newtheorem{defi}[thm]{Definition}
\newtheorem{nota}[thm]{}

\numberwithin{equation}{section}

\newenvironment{sis}{\left\{\begin{aligned}}{\end{aligned}\right.}

\newtheorem{example}[thm]{Example}

\newcommand{\la}{\longrightarrow}
\newcommand{\ha}{\hookrightarrow}

\newcommand{\w}{\widetilde}
\newcommand{\ov}{\overline}
\newcommand{\un}{\underline}

\newcommand{\Div}{\operatorname{Div}}
\newcommand{\Aut}{\operatorname{Aut}}
\newcommand{\Supp}{\operatorname{Supp}}

\newcommand{\codim}{\operatorname{codim}}

\newcommand{\Pic}{\operatorname{Pic}}
\newcommand{\supp}{\operatorname{Supp}}
\newcommand{\cosupp}{\operatorname{Supp}}
\newcommand{\mdeg}{\operatorname{{\underline{de}g}}}
\newcommand{\Set}{{\operatorname{Set}}^1}

\def\L{\mathcal L}
\def\O{\mathcal O}

\def\Y{\mathcal Y}

\newcommand{\G}{\mathbb{G}}

\newcommand{\Z}{\mathbb{Z}}
\newcommand{\R}{\mathbb{R}}
\def\ZZ{\mathcal Z}

\def\me{\underline{e}}
\def\md{\underline{d}}
\def\mo{\underline{0}}

\newcommand{\gS}{\gamma_S}

\newcommand{\nS}{\nu_S}

\newcommand{\pr}[1]{\mathbb{P}^{#1}}

\newcommand{\Xn}{X^{\nu}}
\newcommand{\Xnn}{X'^{\nu}}

 \def\mgb{\overline{M}_g}
\def\Mgb{\overline{M}_g}

\def\agvor{\overline{A}^{Vor}_g}
\def\agale{\overline{A}^{{\text{mod}}}_g}
\def\Agale{\overline{A}^{{\text{mod}}}_g}

\def\tg{{\rm t}_g}
\def\tgvor{\overline{{\rm t}}_g^{Vor}}
\def\tgb{\overline{{\rm t}}_g}

\def\PXg{P^{g-1}_X}
\def\PXgb{\overline{P^{g-1}_X}}
\def\PSd{P_S^{\md}}
\def\PSb{\ov{P_S}}
\def\TSb{\ov{\Theta_S}}
\def\TXg{\Theta(X)}
\def\TSd{\Theta^{\md}_S}

\def\PX'g{P^{g-1}_{X'}}
\def\PX'gb{\overline{P^{g-1}_{X'}}}
\def\TX'g{\Theta(X')}

\newcommand{\tX}{\widetilde{X}}
\newcommand{\YS}{Y_S}

\def\BYS{\Sigma (Y_S)}

\newcommand{\sing}{X_{\text{sing}}}
\newcommand{\sep}{X_{\text{sep}}}

\newcommand{\ST}{\mathcal{ST}}
\newcommand{\SP}{\mathcal{SP}}
\newcommand{\OP}{\mathcal{OP}}
\newcommand{\Del}{{\rm Del}}

\newcommand{\etab}{{\eta_X}}
\newcommand{\etabb}{{\eta_{X'}}}

\newcommand{\Divb}{{\Div^{\mo}}}

\newcommand{\Gg}{{\rm K} _{\gamma}}

\begin{document}

\title{Torelli theorem for stable curves}

\author{Lucia Caporaso}
\address{Dipartimento di Matematica,
Universit\`a Roma Tre,
Largo S. Leonardo Murialdo 1,
00146 Roma (Italy)}
\email{caporaso@mat.uniroma3.it}

\author{Filippo Viviani}
\address{ Dipartimento di Matematica,
Universit\`a Roma Tre,
Largo S. Leonardo Murialdo 1,
00146 Roma (Italy)}
\email{viviani@mat.uniroma3.it}
\maketitle

\begin{abstract}
We study  the Torelli morphism from the moduli space of stable curves to the moduli space of 
principally polarized stable semi-abelic pairs.
We give two characterizations of its fibers, describe its injectivity locus, and 
give a
 sharp upper bound on the cardinality of the finite fibers. We also bound the dimension of the infinite fibers.
\end{abstract}
\tableofcontents

\section{Introduction}

\subsection{Problems and results}
In modern terms, the classical Torelli theorem (\cite{torelli}, \cite{ACGH}) asserts the injectivity of the Torelli map
$ \tg: M_g \to A_g $
from the  moduli scheme  $M_g$, of smooth projective curves of genus $g$, to
the
 moduli scheme $A_g$, of principally polarized abelian varieties of
dimension $g$.

\

\noindent{\it Context.}
It is well known that,  if $g\geq 1$, the schemes     $M_g$ and $A_g$ are not complete;  the problem of
finding good compactifications for them has been thorougly investigated and solved in various ways.
For $M_g$,
the  most  widely studied  compactification   is the moduli space of 
Deligne-Mumford stable curves, $\mgb$.

Now, the Torelli map $\tg$ does not extend to a regular map from $\mgb$ to $A_g$. More precisely,
the largest subset of $\Mgb$ admitting a regular map to $A_g$ extending $\tg$ is the locus of curves of compact type (i.e. every node is a separating node).
Therefore the following  question naturally arises: does there exist a good compactification of $A_g$
which contains the image of an extended Torelli morphism from the whole of $\mgb$?
If so, what are the properties of such an extended map?

It was  known to D. Mumford  
that $\tg$ extends to a morphism
$$\tgvor:\mgb\to \agvor, $$
where $\agvor$ is  the   second Voronoi
toroidal
compactification  of $A_g$; see \cite{AMRT}, \cite{Nam2}, \cite{NamT}, \cite{FC}.
On the other hand,  
 the map $\tgvor$ fails to be injective:
  if $g\geq 3$ it
has positive dimensional fibers over
 the locus
of curves having a separating node (see \cite[Thm 9.30(vi)]{NamT}). Furthermore,
 although $\tgvor$ has finite fibers  away from this  locus, it still fails to be injective
(see \cite{vv}).  The precise  generalization of the Torelli theorem
with respect to the above map $\tgvor$ remains an open problem,
since the pioneering work of   Y. Namikawa.

In recent years,    the space $\agvor$  and   the map $\tgvor$ 
have been placed in a new modular framework by V. Alexeev
(\cite{alex1}, \cite{alex}). As a consequence,
there exists a different   compactification of the Torelli morphism, 
whose geometric interpretation ties in well with the modular descriptions of $\Mgb$ and of the compactified Jacobian.
More precisely,
in \cite{alex1} a new moduli space is constructed,
the   coarse moduli space $\agale$, parametrizing principally polarized
``semi-abelic stable pairs".
The  Voronoi
 compactification $\agvor$   is shown to be  the
normalization of
the   irreducible component of
 $\agale$ containing $A_g$; see Theorem~\ref{alexthm} below.
Next, in \cite{alex}, a new
compactified Torelli morphism, $\tgb$, factoring through $\tgvor$, is defined:
$$\tgb: \mgb\stackrel{\tgvor}{\longrightarrow}\agvor\to \agale.$$
$\tgb$ is   the   map sending a stable curve $X$ to the
principally polarized semi-abelic
stable pair $(J(X)\curvearrowright
\PXgb, \TXg)$. Here
$J(X)$ is the generalized Jacobian of $X$, $\PXgb$ is
a stable semi-abelic variety,  called the compactified Picard scheme
(in degree $g-1$), acted upon by $J(X)$; finally
 $\TXg\subset \PXgb$ is  a Cartier, ample divisor, called the   Theta divisor.
As proved in \cite{alex}, $\PXgb$
  coincides with the previously constructed   compactified
Picard schemes of
\cite{OS}, \cite{simpson},  and \cite{caporaso}; moreover   the definition
of the theta divisor
  extends the classical one very closely.

\

\noindent{\it The main result.}
The goal of the present paper is to establish the precise analogue of the Torelli theorem
for stable curves, using the   compactified Torelli morphism $\tgb$.
This is done in Theorem~\ref{main}, our main result,
which characterizes curves having the same image via $\tgb$.
In particular we obtain that $\tgb$ is injective at curves having 3-edge-connected dual graph
(for example irreducible   curves, or curves with two components meeting in at least three points).
On the other hand $\tgb$ fails to be injective at curves with two components meeting at two points, as soon as $g\geq 5$; see Theorem~\ref{Tinj}.

We actually obtain two different characterizations of curves having the same Torelli image,
one  is based on the classifying morphism
of the generalized Jacobian  (see Section~\ref{Tsec}), the other one, less sophisticated and more explicit,
is of combinatorial type and we shall now illustrate it. 

Let $X$ and $X'$ be two stable curves free from separating
nodes (this is the key case); our main theorem states that  $\tgb(X)=\tgb (X')$ if and only if 
$X$ and $X'$ are ``C1-equivalent", i.e.  if the following holds.
First,  $X$ and $X'$  have the same normalization,
$Y$; let $\nu:Y\to X$ and $\nu':Y\to X'$ be the normalization maps.
Second, $\nu$ and $\nu'$ have the same ``gluing set" $G\subset Y$,
i.e. $\nu^{-1}(\sing)=\nu'^{-1}(X'_{\text{sing}})=G$.
 The third and last requirement is the interesting one, and can only be described after
  a preliminary step: we prove that the set $\sing$ of nodes of $X$
  has a  remarkable partition 
 into disjoint subsets, called ``C1-sets", defined as follows. Two nodes of $X$ belong 
  to the same C1-set if the partial normalization of $X$ at both of them is disconnected. 
  Now, the gluing set $G$ maps two-to-one onto $\sing$ and onto $X'_{\text{sing}}$, so the  partitions of $\sing$ and  of $X'_{\text{sing}}$ in C1-sets, 
   induce each a partition on $G$, which we call the ``C1-partition".
We are ready to complete our main definition:
   two curves are C1-equivalent if their    C1-partitions on $G$ coincide;
   see Definition~\ref{C1set} and Section~\ref{C1part} for details.

Let us explain the close, yet not evident, connection between the C1-sets of $X$ and the compactified Picard scheme $\PXgb$. 
The scheme $\PXgb$ is endowed with a canonical stratification with respect to the action of the Jacobian of $X$. Now, every codimension-one stratum (``C1" stands for ``codimension one")
is isomorphic to the Jacobian of the normalization of $X$ at a uniquely determined C1-set;
moreover, every C1-set can be recovered in this way (although different codimension-one strata may give the same C1-set).

Let us consider  two simple cases.
Let $X$ be irreducible; then no partial normalization of $X$ is disconnected, hence every C1-set has cardinality one.
On the other hand
$\PXgb$ has a codimension-one stratum for every node of $X$. In this case
 the C1-partition   completely determines $X$, as it identifies
all pairs of branches over the nodes; we conclude that   the Torelli map is injective
on the locus of irreducible curves, a fact that, 
for $\tgvor$, 
was well known to   Namikawa.

The next case is more interesting; let $X$ be
a cycle of $h\geq 2$ smooth components, $C_1,\ldots, C_h$,
with $h$ nodes;
then $G=\{p_1,q_1,\ldots, p_h,q_h\}$ with $p_i,q_i\in C_i$. Now every pair of nodes disconnects $X$,
therefore there is only one C1-set, namely $\sing$. On the other hand
the scheme $\PXgb$ is irreducible, and has  a unique codimension-one stratum. 
We obtain that 
all the curves of genus $g$ whose normalization is $\sqcup_1^h C_i$ and whose gluing points
are $\{p_1,q_1,\ldots, p_h,q_h\}$ are C1-equivalent, and hence they  all
have the same image via the Torelli map $\tgb$. This case yields the simplest examples of non-isomorphic curves whose
polarized compactified Jacobians are isomorphic.

\

\noindent{\it Overview of the paper.}
In Section~\ref{C1sec} we state our first version of the Torelli theorem,
and prove a series of useful results of combinatorial type.

The proof of the main theorem,   which occupies Section~\ref{proofsec}, is shaped as follows.
The difficult part is the necessary condition: assume that two curves, stable and free from separating nodes, have the same image, denoted 
$(J\curvearrowright
\ov{P}, \Theta)$,
 under the Torelli map;   we must prove that they are C1-equivalent.
First,
 the structure of $J$-scheme of $\ov{P}$ yields a stratification whose (unique) smallest stratum 
determines the normalization of the  curves, apart from  rational components.
Second, 
the combinatorics of this stratification  
(the $J$-strata form a  partially ordered set, by inclusion of closures)
carries enough information about the combinatorics of the curves, to
determine  the ``cyclic equivalence class" (see \ref{cyceq}) of their dual graphs.
This second part requires a combinatorial analysis, carried out in Section~\ref{C1sec}.
From these two steps one easily obtains that the two curves have the same normalization.
It remains to prove that the gluing sets of the normalization maps are the same, together with their C1-partition.
Here is where we use the   Theta divisor, $\Theta$, its geometry and
its connection with the Abel maps of the curves. See Subsection \ref{rectheta}
for details on this part.

The proof of the converse (i.e. the fact that C1-equivalent curves have the same Torelli image) is based on   the other, above mentioned, characterization of C1-equivalence,
which we temporarily call ``T-equivalence" (the ``T" stands for Torelli).
The crux of the matter is to prove that C1-equivalence and T-equivalence coincide;
we do that in Section~\ref{Tsec}. Having done that,   the proof of the sufficiency  follows directly from the general theory
of compactifications of principally polarized semiabelian varieties,
on which our definition of T-equivalence is based.

The paper ends with a fifth section
where we compute the upper bounds on the cardinality   (Theorem~\ref{Tinj}),
and on the dimension 
(Proposition~\ref{dim-fibers}), of the fibers of $\tgb$.
We prove that  the finite fibers
 have cardinality at most $\left\lceil \frac{(g-2)!}{2} \right\rceil$;
 in particular, since our bound is sharp, we have that, away from curves with a separating node,
    $\tgb$ is injective if and only if $g\leq 4$.
In Theorem~\ref{Tinj}  we give a geometric description of the injectivity locus of $\tgb$.

\

\noindent{\it Acknowledgements.}
We are thankful to Valery  Alexeev for telling us about  Vologodsky's thesis \cite{vv}.
Part of this paper was written while the first author was visiting MSRI, in Berkeley,
for the special semester in Algebraic Geometry;
she wishes to thank the organizers of the program as well as the Institution
for the excellent working conditions and the stimulating atmosphere. We also wish to thank the referee for some useful remarks.

\subsection{Preliminaries}
\label{not}
We work over   an algebraically closed field $k$. A variety over $k$ is
a reduced scheme of finite type over $k$.
A curve is a projective variety of pure dimension $1$.

Throughout the paper $X$ is a connected nodal curve of arithmetic genus
$g$, and  $Y$ is a nodal curve, non necessarily connected.
We denote by $g_Y$ the arithmetic genus of $Y$.

A node $n$ of $Y$ is called a {\it separating
node}
if the number of connected components of $Y\smallsetminus n$ is greater
than the number of connected components of $Y$. We denote by $Y_{\text
{sep}}$
the set of separating nodes of $Y$.

For any subset $S\subset \sing:=\{\text{nodes of }
X\}$, we denote by $\nS:\YS\to X$  the partial normalization of $X$ at
$S$. We denote by  $\gamma_S$ the number of connected components of $\YS$.
The (total) normalization of $X$ will be denoted by
$$
\nu :\Xn\la X=\coprod _{i=1}^{\gamma}C_i
$$
where the $C_i$ are the connected components of $\Xn$.
 The points $\nu^{-1}(\sing)\subset \Xn$ will be often called ``gluing points" of $\nu$.

The dual graph of
  $Y$ will be denoted by $\Gamma_Y$.
The  irreducible components of $Y$ correspond to the vertices of $\Gamma_Y$,
and we shall  systematically identify these two sets.
Likewise we shall identify
the set of  nodes of $Y$  with the set, $E(\Gamma_Y)$, of edges of $\Gamma_Y$.

A graph $\Gamma$   is a {\it cycle}    if
it is connected and has
$h$ edges and $h$ vertices (each of valency $2$) for some $h\geq 1$.
A curve whose dual graph is a cycle will be called a cycle curve.

\begin{nota}{\it The graph $\Gamma_X(S)$ and the graph $\Gamma_X\smallsetminus S$.}
\label{GammaXS} 
Let $S\subset \sing $ be a set of nodes of   $X$;
 we associate to $S$ a graph,
 $
 \Gamma_X(S),
 $   defined as follows. $\Gamma_X(S)$ is obtained from $\Gamma_X$ by contracting to a point every
edge not in $S$. In particular, the set of edges of $\Gamma_X(S)$ is naturally identified with $S$.
Consider $\nu_S:\YS \to X$ (the normalization of $X$ at $S$). Then
  the vertices of $\Gamma_X(S)$ correspond to the
connected components of $\YS$.
For example,   $\Gamma_X(\sing)=\Gamma_X$, and $\Gamma_X(\emptyset)$ is a point.

The graph $\Gamma_X\smallsetminus S$ is defined as the graph obtained from $\Gamma_X$ by removing the edges in $S$ and
leaving everything else unchanged. Of course $\Gamma_X\smallsetminus S$  is equal to the dual graph of $Y_S$.

The above  notation was also used in \cite{CV}.
\end{nota}

\begin{nota}
\label{cyceq}
In graph theory two graphs
   $\Gamma$ and $\Gamma'$ are called {\it cyclically equivalent}
(or ``two-isomorphic"),
 in symbols $\Gamma\equiv_{\rm cyc} \Gamma'$, if there exists a   bijection
  $\epsilon:E(\Gamma)\to E(\Gamma')$ inducing a bijection between the cycles of $\Gamma$ and the cycles of
  $\Gamma'$; such an $\epsilon$ will be called a  {\it cyclic bijection}.
 In other words, if for any orientation on $\Gamma$ there exists an orientation on $\Gamma'$ such that
 the following diagram is commutative:
  $$\xymatrix{
C_1(\Gamma,\Z) \ar[r]^{\epsilon_C}_{\cong} & C_1(\Gamma',\Z)\\
 H_1(\Gamma,\Z)\ar[r]^{\epsilon_H}_{\cong}\ar@{^{(}->}[u] & H_1(\Gamma',\Z) \ar@{^{(}->}[u]
}$$
where the vertical arrows are the inclusions,  $\epsilon_C$ is the (linear) isomorphism induced by $\epsilon$ and $\epsilon_H$
the restriction of $\epsilon_C$ to $H_1(\Gamma,\Z)$.
\end{nota}

\begin{nota}{\emph{The moduli space $\agale$}.}
\label{pairs}

\begin{defi}\cite{alex1}
\label{SSAPP}
A {\it principally polarized stable semi-abelic}
pair (ppSSAP for short) over $k$ is a pair
$( G \curvearrowright P, \Theta)$ where
\begin{itemize}
 \item[(i)] $G$ is a semiabelian variety over $k$, that is an algebraic
group which is an
extension of an abelian variety $A$ by a torus  $T$:
$$1\to T\to G \to A \to 0.$$
 \item[(ii)] $P$ is a seminormal, connected, projective variety of pure
dimension equal to
$\dim G$.
\item[(iii)] $G$ acts on $P$ with finitely many orbits, and with connected
and reduced stabilizers contained in the toric part $T$ of $G$.
\item[(iv)] $\Theta$ is an effective ample Cartier divisor on $P$ which
does not contain
any $G$-orbit, and such that $h^0(P, \O_P(\Theta))=1$.
\end{itemize}
A $G$-variety $(G \curvearrowright P)$ satisfying the first three
properties  above
is called a {\it stable semi-abelic variety}. A
pair $( G \curvearrowright P, \Theta)$
satisfying all the above properties,    except   $h^0(P, \O_P(\Theta))=1$,
is called
a {\it principally polarized stable semi-abelic pair}.

When $G$ is an abelian variety, the word ``semi-abelic" is replaced by
``abelic".
\end{defi}

A  homomorphism
$\Phi=(\phi_0,\phi_1):(G \curvearrowright P, \Theta)\to (G'
\curvearrowright P', \Theta')$
between two ppSSAP is given by   a  homomorphism of algebraic groups
$\phi_0:G\to G'$,
and a morphism $\phi_1:P\to P'$, satisfying the following two conditions:

\item
(1) $\phi_0$ and $\phi_1$ are compatible with the
actions
of $G$ on $P$ and of $G'$ on $P'$.
\item
(2) $\phi_1^{-1}(\Theta')=\Theta$

$\Phi=(\phi_0, \phi_1)$ is an isomorphism if $\phi_0$
and
$\phi_1$ are isomorphisms.

One of the main results of \cite{alex1} is the following
\begin{thm}
\label{alexthm}
There exists a projective scheme $\agale$ which is a coarse moduli space
for
 principally polarized stable semi-abelic pairs.
Moreover the open subset parametrizing   principally polarized stable
abelic pairs
is naturally isomorphic to $A_g$. The normalization of the
irreducible component  of $\agale$ containing $A_g$  (the ``main
component")
is isomorphic to the second toroidal  Voronoi compactification $\agvor$.\end{thm}

To the best of our knowledge, it is not known whether the  main
component of $\agale$ is
normal; see \cite{brion} for an expository account.

\end{nota}

\begin{nota}{\emph{The compactified Torelli map $\tgb:\mgb\to \agale$}.}
\label{torelli-map}

We shall now recall the modular description of
the compactified Torelli map
$\tgb:\mgb\to \agale$.

\begin{defi} \label{balanced}
Let $Y$ be a  nodal curve of arithmetic  genus
$g_Y$.
Let $M$ be a line bundle on $Y$ of multidegree $\md$ and degree $g_{Y}-1$.
We say that $M$, or its multidegree $\md$, is \emph{semistable} if
for every  subcurve
$Z\subset Y$ of arithmetic genus $g_Z$, we have
\begin{equation}\label{BIm}
g_Z -1\leq d_Z,
\end{equation}
where $d_Z:=\deg_ZM$.
We say that $M$, or its multidegree $\md$, is \emph{stable} if
the equality holds in (\ref{BIm}) exactly  for every subcurve  $Z$ which is a
union of connected components
of $Y$. We denote by
$\Sigma(Y)$ the set of stable multidegrees on $Y$.
\end{defi}
We denote by $\Pic^{\md}Y$ the variety of line bundles of multidegree $\md$ on $Y$.
The variety of line bundles having degree $0$ on every component of $Y$,
$\Pic^{\mo}Y=J(Y)$, is identified with the generalized Jacobian.
Using the notation of \ref{not} and \ref{GammaXS}, we now recall   some properties of the compactified Jacobian
$\PXgb$   (see \cite{alex}, \cite{ctheta}).

\begin{fact}
\label{Comp-Pic}
Let $X$ be a connected nodal curve of genus $g$, and $J(X)$ its generalized Jacobian.
\begin{enumerate}[(i)]

\item $\PXgb$ is a SSAV with respect to the natural action $J(X)$.

\item
The orbits of the action of $J(X)$ give a stratification of $\PXgb$:
\begin{equation}\label{strata-comp-pic}
\PXgb =\coprod_{\stackrel{\emptyset\subseteq S \subseteq\sing}{\md \in
\BYS}}\PSd,
\end{equation}
where each stratum $\PSd$ is canonically isomorphic to $\Pic^{\md}\YS$.

\item
\label{sep}
$\BYS$ is not empty if and only if   $\YS  $ has no separating node.
In particular, if $\BYS$ is not empty then
$\sep \subseteq S$.

\item
Each stratum $\PSd$ is a torsor under the generalized Jacobian
$J(\YS)$ of $\YS$,
and the action of $J(X)$ on $\PSd$ factorizes
through the pull-back map
$J(X)\twoheadrightarrow J(\YS)$. Hence every nonempty stratum
$\PSd$   has
dimension
\begin{equation}\label{dimpic}
\dim \PSd=\dim J(\YS)=g-\# S+\gS-1=g-b_1(\Gamma_X(S)).
\end{equation}

\item
\label{cpstr}
If $P^{\md '}_{S'}\subset \overline{\PSd}$ then  $S\subset S'$ and
$\md \geq \md'$
(i.e. $d_i\geq d_i', \  \forall i=1,\ldots, \gamma$).


\item
The smooth locus $\PXg$ of $\PXgb$ consists of the strata of maximal
dimension:
$$
\PXg=\coprod_{\md \in \Sigma(Y_{\sep})}\PSd.
$$
The irreducible components of $\PXgb$ are the
closures of the maximal dimension strata.
\end{enumerate}
\end{fact}

To give the definition of the theta divisor 
we introduce some  notation.
For any multidegree $\md$ on a curve $Y$ and for any $r\geq 0$ we set
\begin{equation}
\label{BN}
W^{r}_{\md}(Y):=\{L\in \Pic^{\md}Y: h^0(Y,L)>r\};
\end{equation}
when $r=0$ the superscript is usually omitted:
$W^{0}_{\md}(Y)=W_{\md}(Y)$.

The normalization  of $X$ at its  set of separating nodes, $\sep$, will be
denoted  by
\begin{equation}
\label{notsep}
\tX=\coprod _{i=1}^{\widetilde{\gamma}}\tX_i
\end{equation}
where the $\tX_i$ are   connected
(and all free from separating nodes). Note that
 $\widetilde{\gamma}=\# \sep+1$.
 We denote by
$\tilde{g}_i$ the arithmetic genus  of $\tX_i$.

The subsequent facts
summarize results of \cite{est}, \cite{alex} and \cite{ctheta}.
\begin{defi}
The theta divisor $\Theta(X)$ of $\PXgb$ is
$$
\Theta(X):=\overline{\bigcup_{\md \in \Sigma(\tX)}
W_{\md}(\tX)}\subset \PXgb .
$$
\end{defi}

\begin{fact}
\label{Theta}

\begin{enumerate}[(i)]
\item
The pair $(J(X)\curvearrowright\PXgb, \TXg)$ is a ppSSAP. In particular
$\TXg$ is Cartier, ample
and $h^0(\PXgb,  \TXg)=1$.

\item

The stratification of $\PXgb$ given by \ref{Comp-Pic}(ii)
induces the stratification
\begin{equation}
\label{strata-theta}
\Theta(X)=\coprod_{\stackrel{\emptyset\subseteq S \subseteq\sing}{\md \in
\BYS}}\TSd,
\end{equation}
where $\TSd:= \{M\in \PSd\: : \: h^0(\YS, M)>0\}\cong W_{\md}^0(\YS)$
is a divisor in $\PSd$.

\item
Let $\YS=\sqcup_{i=1}^{\gamma_S} Y_{i}$ be the decomposition of $\YS$ in
connected components, and  let $\md \in
\BYS$.
The irreducible components of  $\TSd$ are given by
$$(\TSd)_i=\{L\in \PSd\: :\: h^0(Y_{i}, L_{|Y_{i}})>0\},
$$
for every $1 \leq i\leq  \gamma_S$ such that the arithmetic genus of
$Y_i$ is positive.

\end{enumerate}
\end{fact}

\begin{remark}
\label{ss}
From the  description  \ref{Comp-Pic}, we derive  that
there exists
a unique $J(X)$-stratum in $\PXgb$ contained in the closure of every
other
stratum, namely
$$
P_{\sing}^{(g_1-1, \ldots, g_{\gamma}-1)}=\prod_{i=1}^{\gamma}
\Pic^{g_i-1}C_i.
$$
We refer to this stratum as the {\it smallest stratum} of $\PXgb$.
Moreover, according to stratification (\ref{strata-theta}), the restriction of $\TXg$   to the smallest stratum is given by
\begin{equation}
\label{thetasmall}
\TXg_{|P_{\sing}^{(g_1-1, \ldots, g_{\gamma}-1)}}=\bigcup_{i=1}^{\gamma}
\Pic^{g_1-1}C_1\times \cdots\times \Theta(C_i)\times \cdots\times
\Pic^{g_{\gamma}-1}C_{\gamma}.
\end{equation}
\end{remark}

We can now state  the following result of Alexeev (\cite{alex}):

\begin{thm} The classical Torelli morphism   is compactified
by the morphism  $\tgb:\mgb\to \agale$ which maps a stable curve
$X$
to    $(J(X)\curvearrowright \PXgb, \TXg)$. 
 \end{thm}

\end{nota}
\subsection{First reductions}
We shall now show that the ppSSAP $(J(X)\curvearrowright \PXgb, \TXg)$  depends only on the stabilization of every connected
component  of the partial normalization
of $X$ at its separating nodes. Most of what is in this subsection is well known to the experts.

We first recall the notion of    stabilization.
 A connected nodal curve $X$ of arithmetic genus $g\geq 0$ is called \emph{stable} if
each smooth rational component $E\subsetneq X$ meets the complementary subcurve
$E^c=\ov{X\setminus E}$ in at least three points. So, when
 $g=0$ the only stable curve is  $\pr{1}$. If $g=1$ a stable curve is either smooth
 or irreducible with one node. If $g\geq 2$ stable curves are Deligne-Mumford stable curves.

Given any nodal connected curve $X$, the {\it stabilization}   of $X$ is defined as the curve  $\ov{X}$ obtained as follows. If $X$ is stable then 
$X=\ov{X}$; otherwise let $E\subset X$ be an  exceptional component
(i.e.  $E\subsetneq X$ such that $\#E\cap E^c\leq 2$ and $E\cong \pr{1}$), then we contract $E$
to a point, thereby obtaining a new curve $X_1$. If $X_1$ is stable we let $X_1=\ov{X}$, otherwise we choose an exceptional component of $X_1$  and contract it to a point. By iterating this process 
we certainly arrive at a stable curve $\ov{X}$.
It is easy to check that $\ov{X}$  is unique up to isomorphism.

The stabilization of a non connected curve will be defined as the union of the stabilizations of its connected components.

From the moduli properties of  $\Agale$, and the fact that it is a projective scheme, one derives the following useful
\begin{remark}{\it (Invariance under stabilization.)}\label{stab}
Let $X$ be a connected nodal curve of arithmetic genus $g\geq 0$, and let ${\ov{X}}$ be its stabilization.
Then
 $$
 (J(X)\curvearrowright \PXgb, \TXg)\cong
 (J({\ov{X}})\curvearrowright \overline{P_{\ov{X}}^{g-1}}, \Theta({\ov{X}})).
 $$\end{remark}

Now, we show how to deal with separating nodes.
To do that we   must deal with disconnected curves.
Let $Y=\coprod_{i=1}^h Y_i$ be such a curve and $g_Y$ its arithmetic genus,
so that $g_Y=\sum g_{Y_i} -h$.
We have
\begin{equation}
\label{disc}
 \ov{P^{g_Y-1}_Y}=\prod_{i=1}^h\ov{P^{g_{Y_i}-1}_{Y_i}}\  \  \  \text{ and }\   \  \  \Theta(Y)=\bigcup_{i=1}^h\pi_i^*(\Theta(Y_i))
\end{equation}
where $\pi_i: \ov{P^{g_Y-1}_Y}\to \ov{P^{g_{Y_i}-1}_{Y_i}}$ is the $i$-th projection.

The   next Lemma illustrates the recursive structure of $(\PXgb, \Theta(X))$.
For $S\subset \sing$  such that $\BYS$ is non empty (i.e.  $Y_S$ has no separating nodes), denote
\begin{equation}
\label{stratumcl}
\PSb:=\ov{\bigcup_{\md \in \BYS}\PSd}\subset \PXgb \ \   \text{ and  }\  \  \TSb:=\Theta(X)\cap \PSb.
\end{equation}
\begin{lemma}
\label{stratum} Assumptions as above.
There is a natural isomorphism $\PSb \cong \ov{P^{g_{Y_S}-1}_{Y_S}}$,
inducing an isomorphism between $\TSb$ and $\Theta(Y_S)$.
\end{lemma}
\begin{proof}
Recall that $\PXgb$ is  a GIT-quotient, $V_X\stackrel{q}{\la} \PXgb=V_X/G$
where $V_X$ is contained in a  certain Hilbert scheme of curves in projective space
(there are other   descriptions of $\PXgb$ as a GIT-quotient, to  which the subsequent proof can be easily adjusted).
Denote $V_Y:=q^{-1}(\PSb)$ so that $V_Y$ is a G-invariant, reduced,  closed subscheme of $V_X$ and
$\PSb$ is the GIT-quotient
\begin{equation}
\label{VY}
V_Y \la V_Y/G=\PSb.
\end{equation}
The restriction to $V_Y$ of the universal family over the Hilbert scheme is a family of nodal curves
$\ZZ \to V_Y$ endowed with a semistable line bundle $\L \to \ZZ$.
Let $Z$ be any fiber of
  $\ZZ \to V_Y$; then $Z$ has $X$ as stabilization, and the stabilization map $Z\to X$ blows-up
 some set $S'$ of nodes of $X$; note that $S'$ certainly contains $S$.
  Therefore the exceptional divisors corresponding to $s\in S$ form a family over $V_Y$
   $$
  \ZZ \supset  {\mathcal E}_S\to V_Y,
   $$
    such that ${\mathcal E}_S=\coprod_{s\in S}{\mathcal E}_s$ and every  ${\mathcal E}_s$ is a $\pr{1}$-bundle over $V_Y$.
Consider the family of curves obtained by removing ${\mathcal E}_S$:
  $$
  \Y:=\ov{  \ZZ \smallsetminus   {\mathcal E}_S}\to V_Y.
  $$
  By construction the above is a family of nodal curves, all admitting a surjective map
  to $Y_S$ which blows down some exceptional component (over a dense open subset of $V_Y$
  the fiber of $\Y \to V_Y$ is isomorphic to $Y_S$).
  The restriction $\L_{\Y}$ of $\L$ to $\Y$ is a relatively semistable line bundle. Therefore $\L_{\Y}$ determines
 a unique moduli map  $ \mu$ from $V_Y$ to the compactified Picard variety of $Y_S$, i.e.
  $
 \mu:V_Y \to \ov{P^{g_{Y_S}-1}_{Y_S}}.
  $
  The map $\mu$ is of course $G$-invariant, and therefore   it descends to a unique map $\ov{\mu}:V_Y/G  \la \ov{P^{g_{Y_S}-1}_{Y_S}}$.
  Summarizing, we have a commutative diagram
$$\xymatrix{
V_Y\ar@{->>}[rr]^{\mu}
\ar@{->>}[dr]& &  \ov{P^{g_{Y_S}-1}_{Y_S}} \\
& {V_Y/G=\PSb} \ar@{->>}[ru]^{\ov{\mu}}&
}$$

  By Fact~\ref{Comp-Pic} the morphism $\ov{\mu}$ is a bijection. Since  $\ov{P^{g_{Y_S}-1}_{Y_S}}$ is seminormal,
$\ov{\mu}$ is an isomorphism.
Finally, by  Fact~\ref{Theta}   we conclude that   $\ov{\mu}$ maps $\TSb$ isomorphically to   $\Theta(Y_S)$.
  \end{proof}
We say that a ppSSAP $(G\curvearrowright P, \Theta)$  is
\emph{irreducible}
if every  irreducible component of $P$ contains a unique irreducible
component of $\Theta$. In the next result we use the notation (\ref{notsep}).
\begin{cor}
\label{Jacobian-irr}
\begin{enumerate}[(i)]
 \item  If $\sep=\emptyset$   then $(J(X)\curvearrowright \PXgb,$
$ \TXg)$ is irreducible.
\item  In general, we have the decomposition into irreducible non-trivial
ppSSAP:
$$(J(X)\curvearrowright \PXgb, \TXg)=\prod_{\tilde{g}_i>0}
(J(\tX_i)\curvearrowright \overline{P_{\tX_i}^{\tilde{g}_i-1}},
\Theta(\tX_i)).$$
\end{enumerate}
\end{cor}
\begin{proof}
The first assertion follows from \cite[Thm 3.1.2]{ctheta}.
For the second assertion, by \ref{Comp-Pic} we have
$J(X)= \prod_{i=1}^{\widetilde{\gamma}} J(\tX_i)$.
Now we apply Lemma~\ref{stratum} to $S=\sep$.
Note that in this case
$\PSb = \PXgb$, and hence $\TSb=\Theta(X)$.  Therefore
we get
$$
(\PXgb, \Theta(X))\cong (\ov{P_{\tX}^{g_{\tX}-1}}, \Theta(\tX))\cong \prod_{\tilde{g}_i>0} (\overline{P_{\tX_i}^{\tilde{g}_i-1}},
\Theta(\tX_i)).
$$
  \end{proof}

\section{Statement of the main theorem}
\label{C1sec}
\subsection{C1-equivalence.}
\label{poset}
Assume that $\sep =\emptyset$.
We  introduce two
  partially ordered sets (posets for short)  associated to the stratification of $\PXgb$ into
$J(X)$-orbits, described in  (\ref{strata-comp-pic}).

\noindent
$\bullet$ The {\it poset of strata}, denoted
$\ST_X$, is the set  $\{P_S^{\md}\}$   of all strata  of $\PXgb$,
endowed with the following partial order:
\begin{equation}
\label{STeq}
P_S^{\md}\geq P_T^{\me}\Longleftrightarrow \overline{P_S^{\md}}\supset
P_T^{\me}.
\end{equation}

\noindent
$\bullet$
The {\it poset of (strata) supports}, denoted $\SP_X$, is the set of all
subsets $S\subset \sing$ such that the
partial normalization of $X$ at $S$,
$Y_S$, is free from separating
nodes, or equivalently (recall \ref{GammaXS}):
\begin{equation}
\label{SPeq}
\SP_X:=
\{S\subset E(\Gamma_X)| \  \Gamma_X\smallsetminus S  \text{ has no  separating edge} \}.
\end{equation}
Its partial order is defined
as follows:
$$S\geq T \Longleftrightarrow  S\subseteq T.$$
 There is a natural map
$$
\supp_X:\ST_X   \la \SP_X; \  \  \  \
P_S^{\md}   \mapsto S.
$$
$\supp_X$ is order preserving
 (by Fact \ref{Comp-Pic}(\ref{cpstr})), and surjective
(by Fact \ref{Comp-Pic}(iii)).

We have the integer valued function, $\codim$, on $\SP_X$ (cf.  \ref{GammaXS} and (\ref{dimpic})):
\begin{equation}\label{codimS}
\codim( S):=\dim J(X)-\dim J(Y_S)= b_1 (\Gamma_X(S)).
\end{equation}
Notice that $\codim(S)$ is the codimension in $\PXgb$ of every stratum $\PSd\in
\supp_X^{-1}(S)$.
 Moreover $\codim$ is  strictly order reversing.
\begin{lemmadefi}
\label{C1set}  Assume   $\sep =\emptyset$;
let $S\in \SP_X$.
We say that  $S$
 is a {\emph {C1-set}}  if   the two equivalent conditions below hold.
 \begin{enumerate}
\item
\label{C1set1}
 $\codim( S)=1$.
\item
\label{C1set2}
The graph $\Gamma_X(S)$ (defined in \ref{GammaXS}) is a   cycle.
\end{enumerate}
We denote by $\Set X$ the set of all C1-sets of $X$.
\end{lemmadefi}
\begin{proof}
The equivalence between (\ref{C1set1}) and  (\ref{C1set2}) follows from (\ref{codimS}), together with the fact that
for any $S\subset \sing$ the graph $\Gamma_X(S)$  is  connected  and free from separating edges (because the same holds for $\Gamma_X$).
\end{proof}
\begin{nota}
\label{setg}
Under the identification between the nodes of $X$ and the edges of $\Gamma(X)$,
our definition of C1-sets of $X$ coincides  with  that of  C1-sets of $\Gamma(X)$ given  in \cite[Def. 2.3.1]{CV}.
  The set of   C1-sets of
any graph $\Gamma$, which is a useful tool in graph theory,  is denoted     by $ \Set \Gamma$;
  we shall, as usual, identify
 $ \Set \Gamma_X=\Set X$.
The following fact   is   a rephrasing of     \cite[Lemma 2.3.2]{CV}.
\end{nota}
\begin{fact}
\label{C1lm}
Let $X$ be a connected curve free from separating nodes.
\begin{enumerate}
\item
Every node of $X$ is contained in a unique C1-set.
\item
Two nodes of $X$ belong to the same C1-set if and only if the corresponding edges
of  the dual graph $\Gamma_X$ belong to the same cycles of  $\Gamma_X$.
\item
\label{C1lm3}
Two nodes $n_1$ and $n_2$ of $X$ belong to the same C1-set if and only if the normalization of $X$ at $n_1$ and $n_2$ is disconnected.
\end{enumerate}
\end{fact}

\begin{remark}
\label{C1-part} Therefore, if $\sep=\emptyset$ the C1-sets form a partition of $\sing$.  The preimage under the normalization map $\nu$ of this partition is a partition 
 of the set   of gluing points, $\nu^{-1}(\sing)\subset \Xn$.   We shall refer to this partition  of $\nu^{-1}(\sing)$  as the
 {\it C1-partition}.
 \end{remark}
 The main result of this paper, Theorem~\ref{main} below, is based on the following
\begin{defi}[C1-equivalence.]
\label{C1}
Let $X$ and $X'$ be connected nodal curves free from separating nodes;
 denote by $\nu:\Xn \to X$ and $\nu':\Xnn \to X'$  their normalizations.
 $X$ and $X'$ are {\it {C1-equivalent}} if the following conditions hold
\begin{enumerate}[(A)]
\item\label{C1n} There exists an isomorphism $\phi: X^{\nu}\stackrel{\cong}{\to} X'^{\nu}$.
\item \label{C1S}
There exists a bijection between their C1-sets, denoted by
$$
\Set X \to \Set X';\  \  \  S\mapsto S'
$$
such that $\phi(\nu^{-1}(S))=\nu'^{-1}(S')$.
\end{enumerate}

In general,   two   nodal curves $Y$ and $Y'$ are C1-equivalent if there exists
a bijection between their connected components  such that every two corresponding components
are C1-equivalent.
\end{defi}

With the terminology introduced in Remark~\ref{C1-part},  we can informally state that two curves free from separating nodes are C1-equivalent if and only if they have the same
normalization, $Y$, the same set of gluing points $G\subset Y$, and the same C1-partition of $G$.

\begin{example}
\label{C1eg}
\begin{enumerate}
\item
Let $X$ be irreducible. Then for every node $n\in \sing$ the set $\{n\}$ is a C1-set, and every C1-set of $X$ is obtained in this way. 
It is clear that the only curve C1-equivalent to $X$ is $X$ itself.
\item
Let $X=C_1\cup C_2$ be the union of two smooth components meeting at $\delta\geq 3$ nodes (the case $\delta=2$ needs to be treated apart, see below).
 Then again for every $n\in \sing$ we have that $\{n\}\in \Set  X$  so that   $\Set X\cong \sing$.
Also in this case $X$  is the only curve in its C1-equivalent class. The same holds if the $C_i$ have some node.
\item
Let $X$ be such that its dual graph is a cycle of length at least 2. Now the only C1-set is the whole $\sing$
and, apart from some special cases, $X$ will not be the unique curve in its C1-equivalent class;
see  example~\ref{C1ex} and section~\ref{fibsec} for details.
\end{enumerate}
\end{example}
\begin{thm}\label{main}
Let  $X$ and $X'$ be two stable curves of genus $g$.

Assume that $X$ and $X'$ are
free from separating nodes.  Then
 $\tgb(X)=\tgb(X')$
if and only if $X$ and $X'$ are C1-equivalent.

In general, let $\widetilde{X}$ and
$\widetilde{X'}$ be the normalizations of $X$ and $X'$ at their separating nodes.
Then
 $\tgb(X)=\tgb(X')$ if and only if the stabilization of $\widetilde{X}$  is C1-equivalent to the
stabilization of $\widetilde{X'}$.
\end{thm}
By Example~\ref{C1eg} we have that  if $X$ is irreducible, or if $X$
is the union of two components meeting in at least three points,
then the Torelli map is injective (i.e. $\tg^{-1}(\tg(X)=\{X\}$). The locus of curves $X\in  \Mgb$ such that $\tg^{-1}(\tg(X)=\{X\}$ will be characterized in Theorem~\ref{Tinj}. Theorem~\ref{main} will be proved in Section~\ref{proofsec}.

\subsection{Some properties of  C1-sets}
\label{C1part}
Here are a few facts to be applied later.
 
\begin{remark}
\label{C1st}
Let $S\in \Set X$ and
consider $Y_S$, the normalization of $X$ at $S$. By definition $Y_S$ has $\#S$ connected components,
and $\Gamma_X(S)$ can be viewed as the graph whose vertices are the connected components of $Y_S$, and whose edges correspond to $S$.
Since  $\Gamma_X(S)$ is a cycle,  if $X$ is stable  every connected component of $\YS$ has positive arithmetic genus.
 \end{remark}

\begin{lemma}
\label{ST}
Let $S$ and $T$ be two distinct C1-sets of $X$. Then $T$ is entirely contained in a unique connected component of $Y_S$.
\end{lemma}
\begin{proof}
Recall that $Y_S$ has $\#S$ connected components, all free from separating nodes.
By Fact~\ref{C1lm} the set $T$ is contained in the singular locus of $Y_S$.
Let $n_1,n_2\in T$, and  let $X^*$ and $Y^*_S$  be the normalizations at $n_1$  of, respectively, $X$  and $Y_S$.
By Fact~\ref{C1lm}(\ref{C1lm3}) $n_2$ is a separating node of $X^*$ and hence of $Y_S^*$.
Since $Y_S$ has no separating node we get that $n_1$ belongs to the same connected component as $n_2$.
\end{proof} 

In the next Lemma we use  the notations of \ref{setg} and \ref{C1st}.
\begin{lemma}\label{C1dec}
Let $\Gamma$ be an oriented connected graph free from separating edges.
Then the inclusion $H_1(\Gamma,\Z) \subset C_1(\Gamma, \Z) $ factors naturally as follows
$$
H_1(\Gamma,\Z) \ha \bigoplus_{S\in \Set \Gamma}H_1(\Gamma(S),\Z) \stackrel{}{\ha} C_1(\Gamma, \Z)
$$
where the graphs $\Gamma(S)$ have the orientation induced by that of $\Gamma$.
\end{lemma}
\begin{proof}
 Let $S\in \Set \Gamma$ and consider the  natural  map $\sigma_S:\Gamma \to \Gamma(S)$
 contracting all edges not in $S$.
 Recall that $\Gamma(S)$ is a cycle whose set of edges is $S$.
By
  Fact~\ref{C1lm}  we have the following commutative diagram with exact rows
\begin{equation}\label{2graphs}
\xymatrix{
0 \ar[r] & C_1(\Gamma\smallsetminus S,\Z) \ar[r] & C_1(\Gamma, \Z) \ar[r]
& C_1(\Gamma(S),\Z)  \ar[r] & 0 \\
0 \ar[r]  & H_1(\Gamma\smallsetminus S,\Z) \ar[r] \ar@{^{(}->}[u]& H_1(\Gamma, \Z) \ar[r]^{\sigma_{S*}\  }
\ar@{^{(}->}[u] & H_1(\Gamma(S),\Z)  \ar[r] \ar@{^{(}->}[u] & 0 \\
}
\end{equation}
where $\Gamma\smallsetminus S\subset \Gamma$ is the subgraph obtained by removing $S$ from  $E(\Gamma)$.
We claim that  we have the following commutative diagram
\begin{equation}\label{Sgraphs}
\xymatrix{
H_1(\Gamma,\Z) \ar@{^{(}->}[r]^{}\ar@{^{(}->}[d] & \oplus_{S\in \Set \Gamma}H_1(\Gamma(S),\Z)\ar@{^{(}->}[d] \\
C_1(\Gamma,\Z)   &  \ar@{_{(}->}[l] ^{}\oplus_{S\in \Set \Gamma}C_1(\Gamma(S),\Z)
}
\end{equation}
where the vertical arrows are the usual inclusions. The bottom horizontal arrow
is the obvious map mapping an edge $e\in E(\Gamma (S))=S\subset E(\Gamma)$ to itself. It is   injective because two different C1-sets of $\Gamma$ are disjoint
(by \ref{C1lm}) (and surjective as $\Gamma$ has no separating edges). Finally, the top horizontal arrow is the sum
of the maps $\sigma_{S*}$ defined in the previous diagram; it is  injective because the diagram is clearly commutative and the other maps   are injective.
\end{proof}
\subsection{Gluing points and gluing data.}
Let $X$ be such that $\sep=\emptyset$, and let  $S\in \Set X$ be a C1-set   of
cardinality $h$. The partial normalization $Y_S$ of $X$ at $S$
has a decomposition $\YS=\sqcup_{i=1}^h Y_{S,i}$, with   $Y_{S,i}$
connected and free from separating nodes. We denote by $Y_{S,i}^{\nu}$ the normalization of $Y_{S,i}$.
We set
\begin{equation}
\label{GS}
G_S:=\nu^{-1}(S)\subset \Xn.
\end{equation}
Each of the connected components
$Y_{S,i}$ of $Y_S$ contains exactly two of the points in $G_S$, let us call them
$p_i$ and $q_i$.
This enables us to define a unique fixed-point free involution on $G_S$, denoted $\iota_S$,
such that  $\iota_S$ exchanges $p_i$ and $q_i$
for every $1\leq i \leq h$.

 The involutions $\iota_S$ and the curves $Y_{S,i}^{\nu}$
are the same for C1-equivalent curves, by the next result.
\begin{lemma}
\label{detinv}
Let $X$ be   free from separating nodes.
The data of $\Xn$ and of the sets $G_S\subset \Xn$ for every $S\in \Set X$ uniquely determine
the curves $Y_{S,i}^{\nu}\subset \Xn$ and the involution $\iota_S$, for every $S\in \Set X$.
\end{lemma}
\begin{proof}
Pick a C1-set $S$ and let $h=\#S$. Denote $G_S:=\{r_1,\ldots, r_{2h}\}$
and $Y_S=\coprod_1^hY_i$.
We have
\begin{equation}
\label{GSY2}
\#G_S\cap Y_i^{\nu}=2
\end{equation} for every $i$.
Consider the point $r_1$ and 
 call  $Y_1^{\nu}$ the component containing it.
Let us show how to reconstruct $Y_1^{\nu}$.
Let $C_1\subset \Xn$ be the irreducible component containing $r_1$; of course $C_1\subset Y_1^{\nu}$.

Now,  by Lemma~\ref{ST}, for every   $T\in \Set X$ such that $S\neq T$, we have that
if $G_T\cap C_1\neq \emptyset$ then $T$ is entirely contained in the singular locus of $Y_1$.
In particular  every irreducible component of $\Xn$ intersecting $G_T$ is contained in $Y_1^{\nu}$.
Define the following subcurve $Z_1$ of $\Xn=\sqcup C_i$
$$
Z_1:=C_1 \sqcup  \coprod_{\stackrel{\exists T\in \Set X \smallsetminus \{S\}:}{C_i\cap G_T\neq \emptyset, C_1\cap G_T\neq \emptyset}} C_i.
$$
We now argue as before, by replacing $C_1$ with $Z_1$. We get that
if $X$ has a C1-set   $T\neq S$
such that $G_T$ intersects $Z_1$,
then again $T\subset (Y_1)_{\text{sing}}$;
therefore, by   Lemma~\ref{ST}, every component of $\Xn$ intersecting
$G_T$ is contained in $Y_1^{\nu}$.
We can hence inductively define the following subcurve of $Y_1^{\nu}$.
We rename $Z_0:=C_1$; next for $n\geq 1$ we set
$$
 Z_n:=Z_{n-1}\sqcup \coprod_{\stackrel{\exists T\in \Set X \smallsetminus  \{S\}:}{C_i\cap G_T\neq \emptyset, Z_{n-1}\cap G_T\neq \emptyset}} C_i.
$$
Since    all of the  nodes of $Y_1$ belong to some C1-set of $X$,
for $n$ large enough we have
$ 
Z_n=Z_{n+1}=\ldots=Y_1^{\nu}.
$ 
Hence $Y_1^{\nu}$ is uniquely determined.
Now, by (\ref{GSY2}) we have that
 $Y_1^{\nu}\cap G_S=\{r_1, r_j\}$ for a unique  $j\neq 1$; therefore we must have $\iota_S(r_1)=r_j$.
This shows that the curves $Y_{S,i}^{\nu}$ are all determined, and so are the involutions
$\iota_S$.
\end{proof}

\begin{nota}
\label{gldata}
{\it Gluing data of $X$.}
By Lemma~\ref{detinv}, if $X$ and $X'$ are C1-equivalent
for every pair of corresponding C1-sets $S$ and $S'$ the isomorphism between their normalizations preserves the decompositions $\YS=\coprod_{i=1}^h Y_{S,i}$
and $Y'_{S'}=\coprod_{i=1}^h Y'_{S',i}$, as well as the involutions $\iota_S$ and $\iota_{S'}$.
What extra data should one specify   to reconstruct $X$ from its C1-equivalence class?
We now give an answer to this question.
Fix $S\in \Set X$, let $h=\#S$
and $Y_S=\coprod_1^hY_i$. By Lemma~\ref{detinv} the C1-equivalence class of
$X$ determines the involution $\iota_S$ of  $G_S$.
This enables us to write $G_S=\{p_1,q_1,\ldots,p_h,q_h\}$ with $p_i, q_i\in Y_i^{\nu}$.
Of course this is not enough to
determine how $G_S$ is glued on $X$. To describe what is further needed,
we introduce   an abstract set of cardinality $2h$,
denoted $G_h=\{s_1,t_1,\ldots,s_h,t_h\}$, endowed with the
involution $\iota_h$ defined by $\iota_h(s_i)=t_i$ for every $1\leq i\leq h$.

Pick either one of the two cyclic orientations of $\Gamma_X(S)$. We claim that the
gluing data of $G_S$ determine, and are uniquely determined by, the following two items.
\begin{enumerate}
 \item A {\it marking} $\psi_S:(G_h,\iota_h)\stackrel{\cong}{\la}
(G_S, \iota_S),$ where  $\psi_S$ is a bijection 
mapping the (unordered) pair $(s_i, t_i)$ to the pair $(p_i, q_i)$.
 \item A cyclic permutation on $\{1,\ldots, h\}$, denoted by $\sigma_S$, free from fixed points.
\end{enumerate}
Indeed the points $\psi_S(s_i)$ and $\psi_S(t_i)$
correspond, respectively, to the sources and  targets of the orientation of
$\Gamma_X(S)$; the permutation $\sigma_S$ is uniquely determined by the fact
that the point $\psi_S(s_i)$ is glued to the point
$\psi_S(t_{\sigma_S(i)})$.
The opposite  cyclic orientation of $\Gamma_X(S)$ corresponds to changing
\begin{equation}
\label{invgd}
(\sigma_S, \psi_S)\mapsto (\sigma_S^{-1},\psi_S\circ \iota_h);
\end{equation}
 the above transformation defines an involution
on the set of pairs $(\sigma_S,\psi_S)$ as above.
We call the equivalence class $[(\sigma_S,\psi_S)]$, with respect to the above involution,
the \emph{gluing data} of $S$ on $X$.

Conversely, it is clear that a nodal curve $X$ is uniquely determined, within
its C1-equivalence class, by an equivalence class $[(\sigma_S,\psi_S)]$ for each $C1$-set $S\in \Set X$. In fact, $X$ is given as follows
\begin{equation*}
X=\frac{X^{\nu}}{\coprod_{S\in \Set X} \{\psi_S(s_{i})=
\psi_S(t_{\sigma_S(i)})\, :\, 1\leq i\leq \# S\}}.
\end{equation*}

The previous analysis would enable us to explicitly, and easily,
bound the cardinality of any C1-equivalence class.
We postpone this to the final section of the paper; see Lemma~\ref{number-equiv}.
\end{nota}

\subsection{Dual graphs of C1-equivalent curves}
In this subsection, we
shall prove that two C1-equivalent curves have cyclically equivalent dual graphs.
As a matter of fact, we will prove a slightly stronger result.
We first need the following

\begin{defi}\label{strong-equ}
Let $\Gamma$ and $\Gamma'$ be two graphs free from separating edges.
We say that $\Gamma$ and $\Gamma'$ are \emph{strongly cyclically equivalent}
if they can be obtained from one  another
via iterated applications of the following move,
called {\it twisting at a separating pair of edges}:

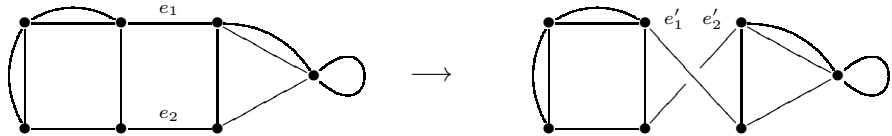
\begin{figure}[!htp]
$$\xymatrix@=1pc{
*{\bullet} \ar@{-}[rr] \ar@{-}[dd] \ar@{-}@/_/[dd] & &*{\bullet} \ar@{-}[dd]
\ar@{-}@/_/[ll]\ar@{-}[rr]^{e_1}& & *{\bullet} \ar@{-}[dd]
\ar@{-}@/^.5pc/[rrd]\ar@{-}[drr]&&&&&&
*{\bullet} \ar@{-}[rr] \ar@{-}[dd] \ar@{-}@/_/[dd] &
&*{\bullet} \ar@{-}[dd] \ar@{-}@/_/[ll] \ar@{-}[ddrr]^<<{e_1'}&&
*{\bullet} \ar@{-}[dd]\ar@{-}@/^.5pc/[rrd]\ar@{-}[drr]&& \\
&&&  &   &&*{\bullet} \ar@{-}@(ur,dr)
&&
\la
&&&&& &    &&*{\bullet} \ar@{-}@(ur,dr)                 \\
*{\bullet} \ar@{-}[rr] & &*{\bullet} \ar@{-}[rr]^{e_2}& &
*{\bullet} \ar@{-}[urr]&&
&&&& *{\bullet} \ar@{-}[rr] &  &*{\bullet} \ar@{-}[uurr]|\hole^>>{e_2'}&&
*{\bullet} \ar@{-}[urr]&&
}$$
\caption{A twisting at the separating pair of edges $\{e_1,e_2\}$.}
\label{twist}
\end{figure}
\end{defi}
The above picture means the following. Since $(e_1, e_2)$ is a separating pair of edges,
we have that $\Gamma \smallsetminus \{e_1, e_2\}$ has two connected components, call them
$\Gamma _a $ and $\Gamma_b$. For $i=1,2$ call $v_i^a$ (resp. $v^b_i$)
the vertex of $\Gamma_a$  (resp. of $\Gamma _b$) adjacent to $e_i$.
Then $\Gamma '$ is obtained by joining the two graphs $\Gamma _a $ and $\Gamma_b$
by an edge $e_1'$ from $v_1^a$ to $v_2^b$ and by another edge $e_2'$ from  $v_2^a$ to $v_1^b$.
Notice that if $v_1^a=v_2^a$  and $v_1^b=v_2^b$, our twisting operation
does not change the isomorphism class of the graph.

\begin{remark}\label{s-cyc-cyc}
If   $\Gamma$ and $\Gamma'$ are strongly cyclically equivalent then they
are cyclically equivalent.

This is intuitively clear. A cyclic bijection $E(\Gamma)\to E(\Gamma ')$ can be obtained by mapping every separating pair of edges at which a twisting is performed to its image. To check that this bijection preserves the cycles
it suffices to observe that if two edges form a separating pair then they belong to the same cycles.
Alternatively, the twisting at a separating pair of edges
is a particular instance of the so-called
second move of Whitney, which does not
change the cyclic equivalence class of a graph (see \cite{Whi}).
\end{remark}

\begin{prop}
\label{equivgr}
Let $X$ and $X'$ be  free from separating nodes and C1-equivalent.
Then $\Gamma_X$ and $\Gamma_{X'}$ are strongly cyclically equivalent
(and hence cyclically equivalent).
\end{prop}

\begin{proof}
By the discussion in \ref{gldata}, it will be enough to show that for every $C1$-set $S\in \Set X$,
any two gluing data  associated to $S$ can be
transformed into one another by a sequence of edge twistings of the type described
 in \ref{strong-equ}.
Moreover, it is enough to consider one $C1$-set at the time,
in fact by \ref{ST}, the twisting
at a separating pair of edges $\{e_1, e_2\}$ belonging to $S\in \Set X$ does
not affect the gluing data of the other $C1$-sets.

So let us fix $S\in \Set X$ of cardinality $h$
and let  $[(\sigma_S, \psi_S)]$ be the gluing data of $S$ on $X$.
We consider two types of edge-twisting, as in \ref{strong-equ}:
\begin{enumerate}
 \item[(a)] Fix a component $Y_j$ of $Y_S$,
 exchange the two gluing points lying on $Y_j$,
$\psi_S(s_j)$ and $\psi_S(t_j)$, and leave everything
 else unchanged.
 On  $\Gamma _X$ this operation corresponds to a
  twisting at the separating
pair of edges of $S$ that join $\Gamma_{Y_j}$ with $\Gamma _{Y_S\smallsetminus Y_j}$
(both viewed as subgraphs of $\Gamma_X$).
The gluing data
are changed according to the rule
$$[(\sigma_S, \psi_S)]\mapsto [(\sigma_S, \psi_S\circ {\rm inv}_j)],$$
where ${\rm inv}_j$ is the involution of   $\{s_1,t_1,\ldots,s_h,t_h\}$
 exchanging $s_j$ with $t_j$ and fixing everything else.
\item[(b)] Fix a connected component $Y_j$ of $Y_S$ and an integer $1\leq a\leq h-1$.
Consider the curve
$$
Z=Y_j \coprod Y_{\sigma_S(j)}\coprod \ldots  \coprod Y_{\sigma_S^a(j)} \subset Y_S.
$$
Now change the gluing data between $Z$ and $Y_S\smallsetminus Z$
by exchanging the two points of $Z$ that are glued to $Y_S\smallsetminus Z$,
and leaving everything else unchanged.
On $\Gamma_X$ this operation corresponds to
 a twisting at the separating pair of edges of $S$ that join
$\Gamma _Z$ to
$\Gamma _{Y_S\smallsetminus Z}$.
The gluing data are changed according to the rule
$$[(\sigma_S, \psi_S)]\mapsto [(\tau_{j,a} \circ \sigma_S \circ
\tau_{j,a}^{-1}, \psi_S\circ {\rm inv}_{j,a})],$$
where $\tau_{j,a}$ is the element of the symmetric group ${\mathcal S}_h$ defined by
$$\tau_{j,a}:=\prod_{0\leq b\leq \lfloor a-1/2 \rfloor}
(\sigma_S^b(j)\sigma_S^{a-b}(j))
$$
and
${\rm inv}_{j,a}$ is the involution of   $\{s_1,t_1,\ldots,s_h,t_h\}$
that exchanges $s_k$ with $t_k$, for all $k=j, \sigma_S(j),\ldots, \sigma_S^a(j)$,
and fixes all the other elements.
\end{enumerate}

The proof consists in   showing that all the possible gluing data of $S$
 can be obtained
 starting from
$[(\sigma_S,\psi_S)]$ and performing  operations of type  (a)  and (b).

First of all  observe that, by iterating  operations of type (a),
it is possible to arbitrarily  modify
the marking $\psi_S$, while keeping  the cyclic permutation $\sigma_S$ fixed.

On the other hand, using the fact
that any two cyclic permutations of the symmetric group ${\mathcal S}_h$ are conjugate, and that ${\mathcal S}_h$ is generated
by transpositions, it will be enough to show that for any transposition $(jk)\in {\mathcal S}_h$,
by iterating operations of type (b),
we can pass from the gluing data $[(\sigma_S,\psi_S)]$ to   gluing data
of the form $[((jk)\circ \sigma_S\circ (jk)^{-1}, \psi'_S)]$ for some marking
$\psi'_S$.
If the transposition $(jk)$ is such that $k=\sigma_S(j)$ (resp. $k=\sigma_S^2(j)$),
then it is enough to apply
the operation (b) with respect to the component $Y_j$ and the integer $a=1$ (resp. $a=2$).
In the other cases, we can write $k=\sigma_S^a(j)$ with $3\leq a\leq h-1$
and then we apply the operation $(b)$ two times:   first with respect to the component
$Y_{\sigma_S(j)}$ and the integer $a-2$;   secondly with respect to the component $Y_j$
and the integer $a$. After these two operations  the cyclic permutation $\sigma_S$
gets changed to $(jk)\circ \sigma_S \circ (jk)^{-1}$ since
$$(jk)=(j \sigma_S^{a}(j))= \tau_{j,a}\circ \tau_{\sigma_S(j), a-2}.$$
\end{proof}
\section{T-equivalence: a second version of the Torelli theorem}
\label{Tsec}
\begin{nota}
\label{basics}
The statement of Theorem~\ref{main} characterizes curves having isomorphic
ppSSAV in terms of their normalization, and of the C1-partition of their gluing points,
determined by the codimension-one strata of the compactified Picard scheme.

In this section we shall give a different characterization, based on the classifying morphism
of the generalized Jacobian.
From the general theory of semiabelian varieties, recall that the generalized Jacobian
$J(X)$ of a nodal curve $X$ is an extension
$$
1\la H^1(\Gamma_X, k^*)=\G_m^{b_1(\Gamma_X)}\la J(X)\la J(\Xn)=\prod_{i=1}^{\gamma}J(C_i)\la 0
$$
(recall that   
  $\sqcup_{i=1}^{\gamma} C_i=\Xn$ is  the normalization of $X$).
  The above extension is determined by the so-called classifying morphism,
  from the character group of the torus $ H^1(\Gamma_X, k^*)$,
  i.e. from $H_1(\Gamma_X, \Z)$, to the dual abelian variety of $J(\Xn)$.
  Since $J(\Xn)$ is polarized by the Theta divisor, its dual variety can be canonically identified with
  $J(\Xn)$ itself.
  So the classifying morphism in our case takes the form
  $$
  c_X:H_1(\Gamma_X, \Z) \la J(\Xn).
  $$
  This morphism $c_X$ will be explicitly described below.  We shall use      the groups of divisors  and line bundles having degree $0$ on every component:
  $$
 \prod_{i=1}^{\gamma} \Div ^0C_i=\Div ^{\mo}\Xn \la \Pic^{\mo}\Xn = \prod_{i=1}^{\gamma} \Pic ^0C_i=J(\Xn).
$$
\end{nota}

\subsection{Definition of T-equivalence}
\begin{nota}
\label{map-delta}
Fix an orientation of $\Gamma_X$   and consider the source and target maps
$$
s,t:E(\Gamma_X)\to V(\Gamma_X).
$$
Now, $s(e)$ and $t(e)$ correspond naturally to the two points of $\Xn$ lying
over the node corresponding to $e$.
We call $s_e, t_e \in \Xn$ such points.
The usual boundary map is defined as follows
$$
\partial: C_1(\Gamma_X,\Z) \to C_0(\Gamma_X,\Z ); \  \  e \mapsto t(e)-s(e)
$$
and  $H_1(\Gamma_X,\Z)=\ker \partial$.
We now introduce   the map
$$
\w{\eta_X}: C_1(\Gamma_X,\Z)  \to \Div^0 X^{\nu} ; \  \
e \mapsto t_e-s_e.
$$
We will denote by $\eta_X$ the restriction of $\w{\eta_X}$ to $H_1(\Gamma_X, \Z)$,
which is easily seen to take values in the subgroup,
$\Div^{\mo}X^{\nu}$, of divisors having  degree $0$ on every component.

Summarizing, we have a commutative diagram
$$\xymatrix{
 H_1(\Gamma_X,\Z) \ar[r]^{\eta_X}\ar@{^{(}->}[d] & \Div^{{\mo}}X^{\nu}\ar@{^{(}->}[d]\\
C_1(\Gamma_X,\Z) \ar[r]^{\w{\eta_X}} & \Div^{0}X^{\nu}.
}$$

The classifying morphism $c_X: H_1(\Gamma_X,\Z)\to J(X^{\nu})$ of $J(X)$
is obtained by composing the homomorphism
$\eta_X: H_1(\Gamma_X, \Z)\to \Div^{\un{0}}X^{\nu}$
with the  quotient map $\Div^{\un{0}}X^{\nu} \to \Pic^{\un{0}}X^{\nu}=J(X^{\nu})$
sending a divisor to its linear equivalence class.
See \cite[Sec. 2.4]{alex} or \cite[Sec. 1.3]{brion}.

\end{nota}

\begin{nota}
\label{autpic}
Recall the set-up and the notation described in \ref{basics}.
 There are automorphisms of   $\Pic^{\mo}\Xn$ and $\Div ^{\mo}\Xn$ that do not change
 the isomorphism class of $J(X)$.
 We need to take those into account.
In order to do that, consider the group $\Gg:=(\Z/2\Z)^{\gamma}$;  note that it
acts diagonally as subgroup of automorphisms,
 $\Gg\ha \Aut(\Div ^{\mo}\Xn)$,
 $\Gg\ha \Aut(\Div \Xn)$,
  and $\Gg\ha \Aut(\Pic^{\mo}\Xn)$, via multiplication by 
  $+1$ or 
  $-1$ on each factor.
 We shall usually identify $\Gg$ with the image of the above monomorphisms.

For example, if $\Xn =C_1\cup C_2$
  then ${\rm K}_2\subset \Aut(\Div^{\mo}\Xn)$ is generated by the involutions $(D_1,D_2)\mapsto (-D_1,D_2)$
and $(D_1,D_2)\mapsto (D_1,-D_2)$.

\end{nota}

\begin{defi}[T-equivalence]
\label{Tdef}
We say that two nodal connected curves $X$ and $X'$ are {\it T-equivalent} if the following conditions hold.
\begin{enumerate}[(a)]
\item\label{Tdef1}
There exists an isomorphism $\phi: X^{\nu}\stackrel{\cong}{\to} X'^{\nu}$ between their normalizations.
\item\label{Tdef2}
$\Gamma_X\equiv_{\rm cyc}\Gamma_{X'}$.
\item \label{Tdef3} For every   orientation on $\Gamma_X$ there exists an
 orientation on $\Gamma_{X'} $ and an automorphism $\alpha\in \Gg\subset \Aut(\Div^{\mo}\Xn)$
 such that
the following diagram commutes
 $$\xymatrix{
 H_1(\Gamma_X,\Z)
\ar[r]^{\etab}\ar[d]_{\cong}^{\epsilon_H}
& \Divb X^{\nu}\ar[d]^{\phi_D\circ\alpha}_{\cong} \\
  H_1(\Gamma_{X'},\Z) \ar[r]^{\etabb} &
\Divb X'^{\nu}\\
}$$
where
 $\epsilon_H$ is defined in \ref{cyceq} and $\phi_D:\Divb X^{\nu}\to \Divb X'^{\nu}$
 is the isomorphism induced by $\phi$.
\end{enumerate}

We say that two non connected nodal curves $Y$ and $Y'$ are T-equivalent if there exists
a bijection between their connected components  such that every two corresponding components
are T-equivalent.
\end{defi}
We shall prove in \ref{Teq} that two curves free from separating nodes are T-equivalent if and only if they are C1-equivalent, thereby getting a new statement of Theorem~\ref{main}.
We first need some observations.
\begin{remark} \label{Tfin}
Let $X$ and $X'$ be T-equivalent and  free from separating nodes.
 Then  part (\ref{Tdef3}) of the definition implies that
$$
 \phi( \nu^{-1}(\sing))=\nu'^{-1}(X'_{\rm{sing}}),
$$
 where  $\Xn \stackrel{\nu}{\la} X$ and $\Xnn \stackrel{\nu}{\la} X'$ are the normalization maps.
 \end{remark}
\begin{remark}
\label{C1cyc}
Suppose that $\Gamma_X$
and $\Gamma_{X'}$ are cyclically equivalent and fix a cyclic bijection
$\epsilon:E(\Gamma_X)\to E(\Gamma _{X'})$. By  \cite[Cor. 2.3.5]{CV}, $\epsilon$ induces a bijection from the C1-sets of $X$ to
those of $X'$, mapping $S$ to $\epsilon (S)$. For this bijection we shall always use the following notation
$$
\Set X \la \Set X';\  \  \  S\mapsto S'.
$$
 \end{remark}

\begin{lemma}
\label{Tind} Let $X$ and $X'$ be T-equivalent connected curves, free from separating nodes; pick a pair of corresponding C1-sets,
$S\in \Set X$ and $S'\in \Set X'$.
Then the normalization of $X$ at $S$ is T-equivalent to the normalization of $X'$ at $S'$.
\end{lemma}
\begin{proof}
Let $Y$ be the normalization of $X$ at $S$ and $Y'$  the normalization of $X'$ at $S'$.
It is obvious that $Y$ and $Y'$  have isomorphic normalizations.
Observe that  $\Gamma_Y=\Gamma_X \smallsetminus S$ and $\Gamma_{Y'}=\Gamma_{X '}\smallsetminus S'$.
The bijection $\epsilon:E(\Gamma_X)\to E(\Gamma_{X'})$
maps the edges of $S$ to the edges of $S'$; hence it
induces a bijection
  $\epsilon_Y:E(\Gamma_Y)\to E(\Gamma_{Y'})$.
To see that $\epsilon_Y$ induces a bijection on the cycles it suffices to
observe that the cycles of $\Gamma_Y=\Gamma_X \smallsetminus S$
are precisely the cycles of $\Gamma_X$ which do not contain $S$ (by Fact~\ref{C1lm}),
and the same holds for $Y'$.
Therefore  $\Gamma_Y$ and $\Gamma_{Y'}$ are  cyclically equivalent.

Finally,  let us pick an orientation  on $\Gamma_Y$ and  extend it to  an orientation on $\Gamma _X$.
The map ${\eta_Y}$ naturally factors
$$
{\eta_Y}:H_1(\Gamma_X\smallsetminus S, \Z)\ha H_1(\Gamma_X, \Z) \stackrel{\etab}{\la}
\Divb \Xn.
$$
Choose  an orientation on $\Gamma_{X'}$ so that condition (\ref{Tdef3}) holds.
Then  we have a commutative diagram
$$\xymatrix{
{\eta_Y}: H_1(\Gamma_Y,\Z) \ar[d]_{\cong}^{}  \ar@{^{(}->}[r] &H_1(\Gamma_X,\Z)
\ar[r]^{\etab}\ar[d]_{\cong}^{\epsilon_H}
& \Divb X^{\nu}\ar[d]^{\alpha}_{\cong} \\
{\eta_{Y'}}: H_1(\Gamma_{Y'},\Z) \ar@{^{(}->}[r]&  H_1(\Gamma_{X'},\Z) \ar[r]^{{\eta_{X'}}} &
\Divb \Xnn. \\
}$$
This proves that condition (\ref{Tdef3}) holds for $Y$ and $Y'$, so  we are done.
\end{proof}
\subsection{C1-equivalence equals T-equivalence}
 \begin{prop}
\label{Teq}
Let $X$ and $X'$ be   connected curves free from separating nodes.
Then $X$ and $X'$ are T-equivalent if and only if they are C1-equivalent.
\end{prop}
\begin{proof}
Suppose that $X$ and $X'$ are T-equivalent. Then property (\ref{C1n}) of Definition~\ref{C1}  obviously holds.
Let us simplify the notation by identifying $X^{\nu}= X'^{\nu}$.
Since the dual graphs of $X$ and $X'$ are cyclically equivalent, we have a cardinality preserving bijection between the C1-sets of $X$ and $X'$, by Remark~\ref{C1cyc}.
To prove part  (\ref{C1S}) of Definition~\ref{C1}  let $S, S'$ be any pair  as in \ref{C1cyc},
and  denote, as usual,
$$
G_S:=\nu^{-1}(S)\subset \Xn \  \  \  \text{ and }\  \  \  G_{S'}:=\nu'^{-1}(S)\subset \Xn.
$$
We must prove that $G_S=G_{S'}$.
Since $X$ and $X'$ are T-equivalent, by  Remark~\ref{Tfin} the gluing sets  are the same:
\begin{equation}
\label{Gsing}
G_{\sing}=G_{X'_{\rm sing}}.
\end{equation}
Let $Y$ be the normalization of $X$ at $S$ and $Y'$ the normalization of $X'$ at $S'$.
By Lemma~\ref{Tind} $Y$ and $Y'$ are T-equivalent.
Now, the normalization of $Y$ and $Y'$ is $\Xn$, and by Remark~\ref{Tfin} applied to $Y$ and $Y'$ we obtain
\begin{equation}
\label{GYsing}
G_{Y_{\rm sing}}=G_{Y'_{\rm sing}}\subset \Xn.
\end{equation}
Now, it is clear that $G_S= G_{\sing}\smallsetminus G_{Y_{\rm sing}}$ and
$G_{S'}=G_{X'_{\rm sing}}\smallsetminus G_{Y'_{\rm sing}}$.
Therefore by (\ref{Gsing}) and (\ref{GYsing}) we get $G_S=G_{S'}$ as wanted.

Conversely,  assume that $X$ and $X'$ are C1-equivalent.
By \ref{equivgr} their graphs are cyclically equivalent.
Let us identify $X^{\nu}= X'^{\nu}$, so that by hypothesis $G_S=G_{S'}$
for every pair of  corresponding C1-sets.
It remains to prove that property (\ref{Tdef3}) of Definition~\ref{Tdef} holds.

We begin with a preliminary definition.
From    \ref{autpic}, recall that the group $\Gg=(\Z/2\Z)^{\gamma}$
acts as subgroup of automorphisms
of $\Div \Xn=\prod_{i=1}^{\gamma} \Div C_i$, by the natural diagonal action
 defined in \ref{autpic} (so that any $\alpha\in \Gg$ acts on each $\Div C_i$
either as the identity or as multiplication by $-1$).
For every $S\in \Set X$ denote as usual $Y_1,\ldots, Y_h$ the connected components
of $Y_S$ and let  $Y_i^{\nu}$ be the normalization of  $Y_i$. We   have
$ \Aut (\Div \Xn)=\prod_{i=1}^h\Aut (\Div Y_i^{\nu})$;  we define a subgroup of $\Gg$
$$
\Gg(S):=\{\alpha \in\Gg \subset \Aut (\Div \Xn):\   \alpha _{|\Div Y_i^{\nu}}=\pm 1 \} .
$$
Let $S$ and $S'$ be  corresponding C1-sets, as above. Let $\Gamma =\Gamma_X$ and $\Gamma '= \Gamma _{X'}$.
The graphs $\Gamma(S)$ and $\Gamma'(S')$ are cycles of length $h=\#S=\#S'$,
whose sets of edges are naturally identified with $S$ and $S'$ respectively.
Hence there is a natural inclusion
$C_1(\Gamma(S),\Z)\subset C_1(\Gamma,\Z)$; ditto for $S'$.
Set (notation in \ref{map-delta})
$$
\w{\eta (S)}:= \w{\eta_X}_{|C_1(\Gamma(S))}:C_1(\Gamma(S))\la \Div^0 \Xn\subset \Div \Xn;\  \  \  e\mapsto t_e-s_e
$$
(where above and throughout the rest of the proof we  omit $\Z$). For any orientation
on $\Gamma(S)$ we let
 $\eta(S)$ be  the restriction of $\w{\eta (S)}$ to $H_1(\Gamma(S))$

\begin{equation}
\label{etacom}
\w{\eta_X}_{|H_1(\Gamma(S))}=\eta(S):H_1(\Gamma(S)) \la \Div  \Xn.
\end{equation}
We define $\eta(S'):H_1(\Gamma'(S')) \to \Div  \Xn$ analogously.
Let us describe $\eta(S)$ and $\eta(S')$.
As $\Gamma(S)$ is a cycle
for any choice of orientation we have a choice of two generators of  $H_1(\Gamma(S))\cong \Z$.
 We pick one of them and call it $c_S$.
Write $G_S=\{p_1,q_1;\ldots ;p_h,q_h\}$ as in (\ref{gldata}).
Up to reordering the components $Y_1,\ldots Y_h$ and switching $p_i$ with $q_i$ we may assume that
\begin{equation}
\label{etaS}
\eta(S)(c_S)=\sum_{i=1}^h(q_i-p_i).
\end{equation}
Notice that the choice of orientation is essentially irrelevant: for any orientation and any generator
$\widetilde{c_S}$ of $H_1(\Gamma(S))$ we have that $\eta(S)(\widetilde{c_S})=\pm\sum_{i=1}^h(q_i-p_i)$.

Similarly, make a choice of orientation for $\Gamma '(S')$ and pick a generator
$c_{S'}$ of $H_1(\Gamma'(S'))$.
Then one easily checks that there exists a partition  $\{1,\ldots, h\}=F\cup G$ in two disjoint sets,  $F$ and $G$,
such that we have

\begin{equation}
\label{etaSS}
\eta(S')(c_{S'})=\sum_{i\in F}(q_i-p_i)+\sum_{i\in G}(p_i-q_i).
\end{equation}

Let $\alpha(S)\in \Gg(S)\subset\Aut(\Div \Xn)$ be the automorphism whose restriction to
 $\Div Y_i^{\nu}$ is the identity
for $i\in F$, and it is multiplication by $-1$ for  $i\in G$.
Now let
$$
\epsilon (S):H_1(\Gamma(S))\stackrel{\cong}{\la}H_1(\Gamma '(S'))
$$ be the isomorphism  mapping $c_S$ to $c_{S'}$.
By construction
   $
\eta(S)=\alpha(S)\circ \eta(S')\circ\epsilon(S),
$
i.e.
the map $ {\eta(S)}$ factors as follows
\begin{equation}
\label{fact}
{\eta(S)}: H_1(\Gamma(S))\stackrel{\epsilon (S)}{\la} H_1(\Gamma(S')) \stackrel{ {\eta(S')}}{\la}   {\Div \Xn}\stackrel{ \alpha(S)}{\la}   {\Div \Xn}.
\end{equation}

We repeat the above  construction for every pair of corresponding C1-sets $(S, S')$.

Using Lemma~\ref{C1dec} and   (\ref{etacom}) we have
$$
\eta_X =\bigr(\oplus_{S\in \Set X}{\eta(S)}\bigl)_{|H_1(\Gamma)}\  \  \  \text{and }\  \  \ \eta_{X'} = \bigr(\oplus_{S'\in \Set X'}  {\eta(S')}\bigl)_{|H_1(\Gamma')}.
$$
Now let
$$
\alpha:=\prod_{S\in \Set X}\alpha(S)\in \Gg
$$
where the product above means  composition of the $\alpha (S)$ is any chosen order.
We claim that for every fixed  $S\in \Set X$ we have
$$
\alpha \circ \eta (S')=\pm \alpha(S)\circ \eta(S').
$$
Indeed,  by \ref{ST}, for  any  $T\in \Set X$,  with  $T\neq S$,
$S$ is entirely contained in the singular locus of a unique connected component
of $Y_T$, call it $Y_{T,1}$.
Therefore the gluing set $G_{S'}=G_S$ is entirely contained in $Y_{T,1}^{\nu}$.
By definition, $\alpha(T)$ acts either as the identity or  as multiplication by $-1$  on
every divisor of $\Xn$ supported on  $Y_{T,1}^{\nu}$; in particular
$\alpha(T)$ acts by multiplication by $\pm 1$ on $\eta(S')(c_{S'})$.
The claim is proved.

As a consequence of this claim and of \ref{fact}
we have
$$
\alpha\circ \eta (S')\circ\epsilon (S)=\pm \eta (S).
$$
Now, if for a certain $S$  the above identity holds with a minus sign on
the right, we change  $\epsilon (S)$  into $-\epsilon (S)$,
but we continue to denote it $\epsilon (S)$ for simplicity.

Using again Lemma~\ref{C1dec} we let
$\epsilon_X:H_1(\Gamma)\stackrel{\cong}{\to} H_1(\Gamma ')$ be the restriction to $H_1(\Gamma)$  of
the isomorphism
$$
\oplus_{S\in \Set X}:\oplus_{S\in \Set \Gamma}H_1(\Gamma(S))\stackrel{\cong}{\to}
\oplus_{S'\in \Set \Gamma'}H_1(\Gamma'(S')).
$$
It is trivial to check that $\epsilon_X$ is an isomorphism. In fact by the proof of Proposition~\ref{equivgr} it is clear that $\epsilon_X$ induces the given bijection
between the C1-sets of $X$ and $X'$.
Combining and concluding, we have a
 a commutative diagram
\begin{equation} \xymatrix{
\eta_X:H_1(\Gamma) \ar@{^{(}->}[r]^{}\ar[d]_{\epsilon_X}^{\cong} & \oplus_{S\in \Set \Gamma}H_1(\Gamma(S))
 \ar[r]^{\  \  \  \oplus\eta(S)}& \Div X^{\nu}\ar[d]_{\alpha}^{\cong} \\
\eta_{X'}:H_1(\Gamma')  \ar@{^{(}->}[r]^{} &  \oplus_{S'\in \Set \Gamma'}H_1(\Gamma'(S'))
 \ar[r]^{\  \  \  \  \  \oplus\eta(S')}& \Div X^{\nu}\\
}
\end{equation}
so we are done.
\end{proof}
 
\section{Proof of the Main Theorem}
 \label{proofsec}
The hard part of the proof of Theorem~\ref{main} is the necessary condition:
  if  two stable curves with no separating nodes 
  have the same image
under the Torelli map, then  
they are C1-equivalent.
The proof  is given in Subsection~\ref{tornec} using the preliminary material of Subsections~\ref{posets} and \ref{rectheta}. 
The proof of the converse   occupies Subsection~\ref{pfsuf}.

\subsection{Combinatorial preliminaries}
\label{posets}
In this subsection we  fix a connected curve $X$ free from separatig nodes, and 
 study the precise relation between the posets $\ST_X$ and $\SP_X$, defined in Subsection~\ref{poset}.
 
We will prove, in Lemma~\ref{equal-posets}, that the support map $\Supp_X:\ST_X\to \SP_X$ is a quotient of posets,
that is,   given $S,T\in \SP_X$ we have $S\geq T$ if and only if  there exists $P_S^{\md}$ and $P_T^{\me}$ in $\ST_X$ such that $P_S^{\md}\geq P_T^{\me}$.
In particular, the poset $\SP_X$ is completely determined by $\ST_X$.
This fact will play a crucial role later on, to recover the combinatorics of $X$ from that of $\PXgb$.

We shall here
  apply some combinatorial results obtained  in \cite{CV}, to which we refer for further details.
First of all, observe that   the  poset $\SP_X$  can be defined   purely in terms of the dual graph
of $X$. Namely $\SP_X$ is equal to the poset $\SP_{\Gamma _X},$ defined in  \cite[Def. 5.1.1]{CV}
as the poset of all $S\subset E(\Gamma _X)$ such that $\Gamma_X\smallsetminus S$ is free
from separating edges, ordered by reverse inclusion.

Next, we need to unravel the combinatorial nature of   $\ST_X$;
recall that its elements  correspond to pairs, $(S,\md )$ where $S\in \SP_X$ and $\md$ is a stable multidegree
on the curve $Y_S$. Now, it turns out that stable multidegrees can be defined in terms of so-called  totally cyclic orientations 
  on the graph  $\Gamma_X$.
To make this precise we introduce a new poset,  $\OP_{\Gamma}$ (cf. \cite[Subsec. 5.2]{CV}).
\begin{defi}
\label{tot}
If $\Gamma$ is a connected graph,
  an orientation of $\Gamma$ is {\it totally cyclic} if   there exists
no proper non-empty subset $W\subset V(\Gamma)$ such that the edges between $W$ and
its complement  $V(\Gamma)\smallsetminus W$ go all in the same direction.

 If $\Gamma$ is not connected,  an orientation is totally cyclic if  its restriction to each connected component
of $\Gamma$ is totally cyclic.

The  poset  $\OP_{\Gamma}$ is defined as the set
$$
\OP_{\Gamma}=\{\phi _S:  \phi _S \text{ is  a totally cyclic orientation on }\  \Gamma\smallsetminus S,\  \forall S\in \SP_{\Gamma}\}
$$
together with the following partial order:
$$\phi_S\geq \phi_T \Leftrightarrow   S\subset  T
\text{ and } \phi_T=(\phi_S)_{|\Gamma\smallsetminus  T}.
$$
\end{defi}
\begin{remark}
It is easy to check that if $\Gamma$ admits some separating edge, then $\Gamma$ admits no totally cyclic orientation.
The converse also holds (see loc. cit).
\end{remark}
\begin{nota}
\label{ordeg}{\it Relation between $\OP_{\Gamma_X}$ and  $\ST_X$.}
How is the poset of 
 totally cyclic orientations related to 
the poset $\ST_X$? This amounts to ask about the connection between totally cyclic orientations and
  stable multidegrees, which  is well known
to be the following.

Pick $Y_S$ and any totally cyclic orientation
$\phi_S$
on $\Gamma _{Y_S}=\Gamma\smallsetminus S$;
for every vertex $v_i$ call   $d^+(\phi_S)_{v_i}$
the number of edges of $\Gamma\smallsetminus S$ that  start
 from   $v_i$ according to $\phi_S$. Now
we define a multidegree $\md (\phi _S)$ on $Y_S$ as follows
\begin{equation}
\label{degor}
\md (\phi _S)_{v_i}:=g_i-1+d^+(\phi_S)_{v_i},\  \  \   i=1,\ldots, \gamma,
\end{equation}
 where $g_i$ is the geometric genus of the component corresponding to $v_i$.
Now:
  
{\it A multidegree $\md$  is stable on $Y_S$ if and only if there exists a totally cyclic orientation
$\phi_S$ such that $\md = \md (\phi _S)$  }  (see  \cite[Lemma 2.1]{beau} and \cite[sec.1.3.2]{ctheta}).

Obviously, two totally cyclic  orientations
define the same multidegree   if and only if the number of edges departing from every vertex is the same.
We shall regard two such   orientations as equivalent:
 \begin{defi}
\label{equiv-or}(The poset ${\ov{\OP_{\Gamma}}}$.)
Two orientations $\phi_S$ and $\phi_{T}$ of $\OP_{\Gamma}$
are equivalent
if $S=T$ and if
$d^+(\phi_S)_v=d^+(\phi_{T})_v$ for every vertex $v$ of $\Gamma$.
The  set of equivalence classes of orientations will be denoted by
${\ov{\OP_{\Gamma}}}$.
The quotient map
$\OP_{\Gamma}  \to \ov{\OP_{\Gamma}}$ induces a unique poset structure on
 ${\ov{\OP_{\Gamma}}}$, such  that two classes $[\phi_S],  [\phi_T]\in {\ov{\OP_{\Gamma}}}$
 satisfy
$[\phi_S] \geq [\phi_T]$ if there exist respective representatives
$\phi_{S}$ and $\phi_{T}$ such that
$ \phi_{S}\geq \phi_{T}$ in $\OP_{\Gamma}$.

\end{defi}
The above definition coincides with \cite[Def. 5.2.3]{CV}.

\end{nota}
We shall soon prove that there is a natural isomorphism of posets between ${\ov{\OP_{\Gamma_X}}}$ and $\ST_{\Gamma_X}$.
Before   doing that, we recall the  key result   about the relation between ${\ov{\OP_{\Gamma_X}}}$ and $\SP_{\Gamma_X}$. 
\begin{fact}
\label{quot} Let $\Gamma$ be a connected graph free from separating edges; consider the
   natural maps
$$\begin{aligned}
 \cosupp_{\Gamma}:\OP_{\Gamma} & \la &\ov{\OP_{\Gamma}} &\stackrel{\ov{\cosupp_{\Gamma}}}{\la} & \SP_{\Gamma}\\
\phi_S&  \mapsto &[\phi_S]&\mapsto &S.\  \  \
\end{aligned}
$$
 \begin{enumerate}\item
The maps $ \cosupp_{\Gamma}$ and $\ov{\cosupp_{\Gamma}}$ are quotients of posets.
 \item
 \label{quot2}
 The poset $\SP_{\Gamma}$ is completely determined, up to isomorphism, by   the poset
 $\ov{\OP_{\Gamma}}$ (and conversely). 
\end{enumerate}
\end{fact}
By \cite[Lemma 5.3.1]{CV} the map $ \cosupp_{\Gamma}$ is a quotient of posets, hence so is $\ov{\cosupp_{\Gamma}}$ 
(as $\OP_{\Gamma}  \to \ov{\OP_{\Gamma}}$ is a quotient of posets by definition).
Part  (\ref{quot2}) is the equivalence between {\it (iii)} and {\it (v)} in \cite[Thm 5.3.2]{CV}.
  
\

Now, as we explained in \ref{ordeg}, to every
$\phi_S\in \OP_{\Gamma_X}$ we can associate a stable multidegree $\md(\phi_S)$ of $Y_S$ (see (\ref{degor}));
moreover two equivalent orientations define the same multidegree. This enables us to 
define two maps,  $st_X$ and ${\ov{st_X}}$, as follows
\begin{equation}\label{deg-map}
\begin{aligned}
st_X: \OP_{\Gamma_X} & \la &\ov{\OP_{\Gamma_X}} &\stackrel{\ov{st_X}}{\la} &\ST_X \\
\phi_S & \mapsto  &[\phi_S]&\mapsto &P_S^{\md(\phi_S)}.
\end{aligned}
\end{equation}
\begin{lemma}\label{equal-posets}
Let $X$ be   connected and  free from separating nodes.
Then 
\begin{enumerate}
\item
\label{OPT}
the map  $\ov{st_X}: {\ov{\OP_{\Gamma_X}}} \la \ST_X$
is an isomorphism of posets;
\item
\label{SPT}
there is a commutative diagram
 \begin{equation}\label{dia-posets}
\xymatrix{
\OP_{\Gamma_X} \ar@{->>}[rr]^{st_X} \ar@{->>}[dd]_{\cosupp_{\Gamma_X}} \ar@{->>}[dr]& &
\ST_X \ar@{->>}[dd]^{\Supp_X}  \\
& {\ov{\OP_{\Gamma_X}}} \ar[ru]_{\ov{st_X}}^{\cong} \ar@{->>}[dl]^{\ov{\cosupp_{\Gamma_X}}}&\\
\SP_{\Gamma_X} \ar@{=}[rr] & & \SP_X
}\end{equation}
where every map is a quotient of posets.
In particular
the poset $\SP_X$ is completely determined 
(up to isomorphism)
by the poset $\ST_X$.
\end{enumerate}\end{lemma}
 \begin{proof}
The maps $st_X$ and $\ov{st_X}$ are surjective by what we said in \ref{ordeg}.
 Moreover, by \cite[Prop. 5.1]{caporaso}, they are
 morphisms of posets. From the definitions (\ref{degor}) and \ref{equiv-or} it is clear that
 $\ov{st_X}$ is bijective, and hence an isomorphism of posets.

The commutativity of the diagram is clear by what we said above.
Finally, by Fact~\ref{quot}(\ref{quot2}) we know that  $\ov{\OP_{\Gamma_X}}$ completely determines  $ \SP_X$ 
  as poset, hence part (\ref{SPT}) follows  from part (\ref{OPT}).
 \end{proof}

\subsection{Recovering gluing points from the Theta divisor}
\label{rectheta}
\begin{lemma}
\label{cod1desc}
 Let $P^{\md}_S$ be a codimension-1 stratum of $\PXgb$
and let $h$ be the number of irreducible components of $\Theta(X)\cap P^{\md}_S$.
Then $S$ is a $C1$-set  of cardinality h.
\end{lemma}
\begin{proof}
We have already proved most of the statement in \ref{poset}.
The only part that needs to be justified is the one concerning $\Theta(X)$.
By \ref{C1st}
every connected component of $\YS$ has positive genus.
Now, according
to Fact \ref{Theta}(iii), the number $h$ of irreducible components
of $\Theta(X)\cap P^{\md}_S$ is equal to $\gamma_S=\#S$.
\end{proof}

\begin{nota}
 \label{bipglu}
The following set-up will be fixed throughout  the rest of this subsection.
 $X$ is a stable curve of genus $g$, $\sep=\emptyset$, and $S\in \Set X$ is a $C1$-set   of
cardinality $h$. As usual
$
\nu_S:\YS\to X
$
denotes the normalization at $S$.
We have
 $\YS=\coprod_{i=1}^h Y_i$, with   $Y_i$
  connected,  of arithmetic genus $g_i>0$,
free from separating nodes. The gluing set of $\nu_S$ is denoted
 $\{p_1,q_1,  \ldots , p_h,q_h\}$
with
$\nu_S(p_j)=\nu_S(q_{j+1})$ and
 $p_j,q_j\in Y_j$.

The pull-back via the partial normalization $\nu_S$ induces an exact sequence
$$
0\to k^*\to \Pic X \stackrel{\nu_S^*}{\longrightarrow} \Pic \YS=\prod_{i=1}^h \Pic Y_i\to 0.
$$
\end{nota}
In the following statement we use notation (\ref{BN}).

\begin{lemma}\label{pb}
Fix $\md \in \Sigma(X)$. Let $M$ a general line bundle in $\Pic^{\md}\YS$, $M_i:=M_{|Y_i}$
and $\un{d_i}:=\mdeg M_i$. Let $y_i$ be a fixed smooth point of $Y_i$. Then for    $i=1, \ldots, h$ the following properties hold.
\begin{enumerate}[(i)]
\item
\label{pb1}
$h^0(Y_i,M_i)=1$ (hence $h^0(Y_S,M)=h$).
\item
\label{pbs}
Set $\un{d_i}(-y_i):=\mdeg M_i(-y_i)$. Then $\un{d_i}(-y_i)$ is semistable.
\item
\label{pb2}
$M_i$ does not have a base point in $y_i$
(i.e. $h^0(Y_i,M_i(-y_i))=0$).
\item
\label{pb3} The restriction of the pull-back map, $\nu_S^*:W_{\md}(X)\la \Pic^{\md}\YS$, is birational.
\item
\label{pb4}
$\dim W_{\un{d_i}}^1(Y_i)\leq g_i-2$
for every $1\leq i\leq h$.
\item
\label{pb6}
For any point $p_k$, define
\begin{equation}
\label{Tq1}
T_{p_k}=\left\{M\in \Pic^{\md}\YS \smallsetminus W_{\md}^h(\YS): \:
\begin{aligned}
&h^0(\YS, M(-p_k))=h \text{ and } \\
&h^0(\YS, M(-q_j))<h, \  \forall 1\leq j\leq h.
\end{aligned}
\right\}.
\end{equation}
Define $T_{q_k}$ by replacing   $p_k$ with $q_k$ and $q_j$ with $p_j$
in (\ref{Tq1}).
Then
$$\Pic^{\md}\YS\setminus \nu_S^*(W_{\md}(X))=\bigcup_{k=1}^h (T_{p_k}\cup T_{q_k}).
$$
\end{enumerate}
\end{lemma}
\begin{proof}
Since $\md$ is stable, Theorem 3.1.2 of \cite{ctheta} yields that $W_{\md}(X)$ is irreducible of
dimension $g-1$.

For any $M\in \Pic \YS$ we set $F_M(X):=\{L\in \Pic X:\nu^*L=M\}\cong k^*$.

To prove (\ref{pb1}), observe that the stability of $\md$
yields $\deg(M_i)=g_i$. Therefore the theorem of Riemann-Roch
gives   $h^0(Y_i, M_i)\geq 1$.
Suppose, by contradiction, that $h^0(Y_i, M_i)>1$ for some $i$.
Then $h^0(\YS,M)=\sum_{i=1}^h h^0(Y_i, M_i)\geq h+1$ for every $M\in \Pic^{\md}\YS$.

This implies that $F_M(X)\subset W_{\md}(X)$ (indeed there are at most $h$ conditions
on the global sections of $M$ to descend to a global section of a fixed $L\in F_M(X)$).
Therefore
$$
\dim W_{\md}(X)=\dim \Pic^{\md}\YS +1=\sum_{i=1}^h g_i+1=g
$$
a contradiction. This proves (\ref{pb1}).

For  (\ref{pbs}) and (\ref{pb2}),
set
$\md_i'=\un{d_i}(-y_i)$; observe   that   $|\md_i'|=g_i-1$. Let $Z\subset Y_i$ be a subcurve of $Y_i$ and
let $\w{Z}:=\nu(Z)\subset X$.
Then, of course, $g_Z\leq g_{\w{Z}}$.
Denoting $d'_{i,Z}=|(\md_i')_{|Z}|$ the total degree of $\md_i'$ restricted to $Z$, and by
$d_{\w{Z}}=|\md _Z|$, we have

\begin{displaymath}
d'_{i,Z}=\left\{\begin{array}{ll}
d_{\w{Z}}\geq g_{\w{Z}}\geq g_Z &\text{ if } y_i\not\in Z,\\
d_{\w{Z}}-1\geq g_{\w{Z}}-1\geq g_Z-1 &\text{ if } y_i\in Z,\\
\end{array}\right.
\end{displaymath}
where we used that $d_{\w{Z}}\geq  g_{\w{Z}}$  ($\md$ is stable). So (\ref{pbs}) is proved.
We can therefore use a result due to A. Beauville (\cite{beau}, see also Proposition 1.3.7 in \cite{ctheta}),
stating that every irreducible component of $W_{\un{d_i}'}(Y_i)$ has dimension equal to
$g_i-1$, and in particular that $W_{\un{d_i}'}(Y_i)\neq \Pic^{\un{d_i}'}(Y_i)$.
 Therefore for the general $M\in \Pic^{\md}(\YS)$, we have that $h^0(Y_i, M_i(-y_i))=0$
and this proves part (\ref{pb2}).

In order to prove (\ref{pb3}), we need to make  the isomorphism
$F_M(X)\cong k^*$ explicit. Any $c\in k^*$ determines a unique
$L^c\in F_M(X)$, defined as follows.
For every $j=1,\ldots ,h$ consider the two fibers of $M$ over
 $p_j$ and $q_{j+1}$
(with $q_{h+1}=q_1$ as usual, recall that  $\nu_S$ glues $p_j$ with $q_{j+1}$)
and fix an isomorphism between them. Then $L^c\in F_M(X)$ is obtained by
gluing $M_{p_j}$ to $M_{q_{j+1}}$  via the isomorphism
$$
\begin{sis}
& M_{p_j}\stackrel{\cong}{\rightarrow} M_{q_{j+1}} & \text{ for } j=1, \ldots, h-1 , \\
& M_{p_h}\stackrel{\cdot c}{\rightarrow} M_{q_1} &  \\
\end{sis}$$
where the last isomorphism is given by multiplication by $c$.
Conversely, every $L\in F_M(X)$ is of type $L^c$, for a unique $c\in k^*$.

Now, by
(\ref{pb1}) we known that a general $M\in \Pic^{\md}(\YS)$ does not belong
to $W_{\md}^h(\YS)$, i.e. we have $h^0(Y_i, M_i)=1$ for all $i=1,\ldots,h$.
Take a generator, call it $\alpha_i$, of  $H^0(Y_i, M_i)$ and set $a_i^p:=\alpha_i(p_i)$ and $a^{q}_i:=\alpha_i(q_{i})$.
A section $\alpha=\sum_{i=1}^h x_i \alpha_i\in H^0(\YS, M)$ descends to a section of
$L^c\in F_M(X)$ on $X$ if and only if it verifies the following system of equations:
\begin{equation}\label{syst}
\begin{sis}
& x_i a_i^p=\alpha(p_i)=\alpha(q_{i+1})=x_{i+1}a_{i+1}^q & \text{ for }
1\leq i\leq h-1, \\
& c x_h a_h^p=c \alpha(p_h)=r(q_1)=x_1a_1^q. &
\end{sis}
\end{equation}
The above system of $h$ equations in the $h$ unknown $x_1, \ldots, x_h$
admits a non-zero solution if and only if the determinant of the associated matrix is zero,
that is if and only if
\begin{equation}\label{deter}
c \cdot \prod_{i=1}^ha_i^p= \prod_{i=1}^h a_i^q.
\end{equation}
Since a general $M\in \Pic^{\md}(\YS)$ verifies $a_i^p\neq 0$ and $a_i^q\neq 0$
for every $i$ (by part (\ref{pb2})), the above equation has
a unique solution $c$ and therefore $F_M(X)$ has a unique point in $W_{\md}(X)$.
This proves (\ref{pb3}), since   $\dim W_{\md}(X)=\dim \Pic^{\md}(\YS)=g-1$.

Now we prove (\ref{pb4}). The fiber of the birational map $\nu^*:W_{\md}(X)\to \Pic^{\md}\YS$
over $ W^h_{\md}(\YS)$ has dimension $1$; hence, as $W_{\md}(X)$  is irreducible of dimension
$\sum_{i=1}^h g_i$, we have
 $\dim W^h_{\md}(\YS)\leq \sum_{i=1}^h g_i-2$. Since $W^h_{\md}(Y)=\bigcup_{i=1}^h
(\pi_i)^{-1}(W_{\un{d_i}}^1(Y_i))$, where $\pi_i:\Pic^{\md}(\YS) \to \Pic^{\un{d_i}}(Y_i)$ is the projection,
we deduce that $\dim W_{\un{d_i}}^1(Y_i)\leq g_i-2$.

Finally (\ref{pb6}). As observed before, we have
$$
\Pic^{\md}\YS\smallsetminus \nu_S^*(W_{\md}(X))
\subset \Pic^{\md}\YS\smallsetminus W^h_{\md}(\YS).
$$
 With the above notation, a line bundle
$ M\in \Pic^{\md}\YS\smallsetminus W^h_{\md}(\YS)$ does not belong to $\nu^*(W_{\md}(X))$
if and only the equation (\ref{deter}) does not admit a solution $c\in k^*$.
This happens precisely when either $a^p_k=0$ for at least one $k$ and
$a^q_i\neq 0$ for any $i$, or if $a^q_k=0$ for  at least one $k$ and
$a^p_i\neq 0$ for any $i$. These conditions are easily seen to be equivalent to
the fact that $M\in \cup_{k}(T_{p_k}\cup T_{q_k})$.
\end{proof}

\begin{prop}\label{theta}
Let $X$ be such that $\sep = \emptyset$; pick $S\in \Set X$ and
 $\md\in \Sigma(X)$.
The image of the pull-back map
$\nu_S^*:W_{\md}(X)\to \Pic^{\md}\YS$ uniquely determines $\nu_S^{-1}(S)$,
the gluing set of $\nu_S$.
\end{prop}
\begin{proof}
Denote
$ \Pic^{\md}\YS\smallsetminus \nu_S^*(W_{\md}(X))=T_1\coprod T_2
$ where, using  Lemma~\ref{pb} (\ref{pb6}) we have
\begin{equation}
\label{T12}
T_1:=\cup_{k=1}^hT_{p_k},\  \  \  T_2:=\cup_{k=1}^hT_{q_k},
\end{equation}
for some set $\{p_1, \ldots, p_h, q_1,\ldots, q_h\}$ which we must prove is uniquely determined, up to reordering the $p_i$  (or the $q_i$)
among themselves.
Notice that, for any such set,
two different points
 $p_k, p_j$ lie in two different connected components  of $Y_S$, and the same holds
for any  two $q_k, q_j$. Therefore
 $T_1$ and $T_2$ are connected; on the other hand they   obviously
 do not intersect, therefore they are determined.
It thus suffices to  prove that $T_1$ (and similarly $T_2$)  determines a unique set of $h$ smooth points of $Y_S$
such that $T_1$ is expressed as in (\ref{T12}).

We begin with a preliminary analysis.
Pick any smooth point  of $Y_S$, let $Y_k$ be the connected component on which it lies,
name the point $y_k$, for notational purposes.
By Lemma~\ref{pb}(\ref{pbs}) the multidegree $\md_k':=\md_k(-y_k)$ is semistable on $Y_k$.
Therefore we can  apply
 Proposition 3.2.1 in \cite{ctheta}. This yields   that $W_{\md'_k}(Y_k)$ contains an irreducible component
(of dimension $g_k-1$) equal to the image of the $\md'_k$-th Abel map;
we call $A_k$ this component. We also have  (by loc. cit.) that  $A_k$ does not have a fixed base
point\footnote{$V\subset \Pic Y$
  has a fixed base point if there exists a   $y\in Y$
which is a base point for every $L\in V$.}, and that  $h^0(Y_k,L)=1$ for the general $L\in A_k$.

We can thus define an irreducible effective divisor, as follows
$$
D_{y_k}:=\{M\in \Pic^{\md}\YS\: |\: M_k(-y_k)\in A_k\}.
$$
Observe that $D_{y_k}$ has no fixed base point other than $y_k$. Indeed, let $M\in D_{y_k}$ be a general point.
If $j\neq k$ then $M_j$ is general in $\Pic^{\md_j}Y_j$, hence by \ref{pb} $M_j$ has no fixed base point and $h^0(Y_j,M_j)=1$.
On the other hand $M_k$ varies in a set of dimension $g_k-1$, therefore $h^0(Y_k,M_k)=1$
by \ref{pb}(\ref{pb4}). Therefore, if every $M_k$ had a base point in $r\neq y_k$, we would  obtain
\begin{equation}
\label{pr}
1=h^0(M_k)=h^0(M_k(-y_k))=h^0(M_k(-r))=h^0(M_k(-y_k-r)).
\end{equation}
But $M_k(-y_k)\in A_k$, so every element of $A_k$ would have a base point in $r$, which is not possible (see above).

Summarizing,   the general $M\in D_{y_k}$ satisfies the following properties
\begin{equation}
\begin{sis}
\label{prop-div}
& h^0(Y_j, M_j)=1 &\text{ for any } j=1, \ldots, h.\\
&h^0(\YS,M)=h^0(\YS,M(-y_k))=h \\
& h^0(\YS, M(-r))<h^0(\YS, M) &  \forall r\neq y_k \text{ smooth point of }\YS.
\end{sis}
\end{equation}
Now, back to the proof of the proposition; it suffices to concentrate on $T_1$. By contradiction,   suppose   there are two different
descriptions for $T_1$ as follows
 $$
T_1=\bigcup_{j=1}^hT_{p_j}=\bigcup_{j=1}^hT_{\tilde{p}_j};
$$
 we may assume
$\tilde{p}_1\not\in \{p_ 1,\ldots, p_h\}$.
By (\ref{prop-div}) applied to $y_k=\tilde{p}_1$, together with \ref{pb}(\ref{pb6}), we have
$$
D_{\tilde{p}_1}\subset \overline{T_1}.
$$
But then, since $T_1=\cup_{j}T_{p_j}$,  we conclude that $D_{\tilde{p}_1}$ has a fixed base point in some $p_j$, which is impossible by the last property in (\ref{prop-div}).
\end{proof}

\subsection{Torelli theorem: proof of the necessary condition}
\label{tornec}

By Corollary~\ref{Jacobian-irr} and Remark~\ref{stab},  to prove the necessary condition of Theorem~\ref{main} it suffices to prove the following.

{\it Let $X$ and $X'$ be stable curves  of genus $g$ free from separating nodes, and such that $\tgb(X)=\tgb(X')$.
Then $X$ and $X'$ are C1-equivalent.}

So, suppose we have an isomorphism
$$\Phi=(\phi_0,\phi_1):(J(X) \curvearrowright \PXgb, \Theta(X))\stackrel{\cong}{\longrightarrow} (J(X')
\curvearrowright \ov{P_{X'}^{g-1}}, \Theta(X')).$$
We divide the proof into Steps. In the first step we collect  
 the combinatorial parts.

\

\noindent
{\bf{Step 1.}}   \begin{enumerate} \item
{\it   The above isomorphism $\Phi$ induces  a bijection
$$
\Set X \la \Set X';\  \  S\mapsto S',
$$
such that $\#S=\#S'$ for every $S\in \Set X$.
 \item
 $\Gamma $ and $\Gamma'$ are cyclically equivalent.  }
\end{enumerate}

The isomorphism $\phi_1:\PXgb\stackrel{\cong}{\la} \PX'gb$ induces an isomorphism between the
posets of strata $\ST_X\cong \ST_{X'}$; hence, by Lemma~\ref{equal-posets}, it induces an
isomorphism 
$$
\SP_X\cong \SP_{X'}
$$ of the posets of supports, compatible with the support maps.
In particular, we have an induced bijection 
$$
\Set X \la \Set X';\  \  S\mapsto S'.
$$
Let us show that this bijection preserves cardinalities.
By what we just said, every      stratum of type
$P_S^{\un{d}}$  is mapped
isomorphically to a stratum of type  $P_{S'}^{\un{d'}}$.
Moreover, as the theta divisor of $X$ is mapped isomorphically to
the theta divisor of $X'$, the intersection
$\Theta(X)\cap P^{\md}_S$ is mapped isomorphically to
$\Theta(X')\cap P^{\md'}_{S'}$;
in particular the number of irreducible components of these two intersections is the same. Hence, by Lemma~\ref{cod1desc},
 $S$ and $S'$ have the
same cardinality.

This proves the first item. At this point the fact  
$\Gamma_X$ and $\Gamma_{X'}$ are cyclically equivalent
follows immediately by what we just proved, thanks to the following
immediate  consequence of \cite[Prop. 2.3.9 (ii)] {CV} combined with \cite[Thm 5.3.2(i) - (iii)]{CV}.
\begin{fact} 
Let $\Gamma$ and $\Gamma'$ be two connected graphs free from separating edges.
Suppose that there exists an isomorphism of posets, $\SP_{\Gamma}\cong \SP_{\Gamma'}$,
whose restriction to C1-sets, $\Set \Gamma \cong \Set \Gamma'$,
preserves the cardinality. Then $\Gamma $ and $\Gamma'$ are cyclically equivalent.
\end{fact}

\noindent
{\bf{Step 2.}}
  \label{norm}
  $X^{\nu}\cong X'^{\nu}$.

By the previous step,  the number of irreducible components of $\Xn$ and $X'^{\nu}$ is the same;
indeed,
 the number of edges and the first Betti number of $\Gamma_X$ and $\Gamma_{X'}$ are the same,
 hence the number of vertices is the same.
 Denote by
 $\Xn_+\subset \Xn$ and   $X'^{\nu}_+\subset X'^{\nu}$  the union of all components of positive genus.
It is enough to show that
\begin{equation}
\label{+}
\Xn_+\cong X'^{\nu}_+.
\end{equation}
In Remark~\ref{ss} we saw  that $\PXgb$ has    a unique stratum of smallest dimension,
namely the unique stratum supported on $\sing$.
This smallest stratum is isomorphic to the product of the Jacobians of the components of $\Xn$,
and hence to the product of the Jacobians of the components of $\Xn$ having positive genus.
It is clear that the smallest stratum of $\PXgb$ is mapped by $\phi_1$ to the smallest stratum of
$\PX'gb$. Recall now   (\ref{thetasmall}),
expressing the restriction
of the Theta divisor to this smallest stratum
in terms of the Theta divisors of the components of $\Xn$.
As a consequence the projection of the smallest stratum onto each of its factors
determines the polarized Jacobian of all the positive genus components of
the normalization. Hence, by the Torelli theorem for smooth curves, we obtain that
the positive genus components of the normalizations of $X$ and $X'$ are isomorphic,
so (\ref{+}) is proved.

\noindent
{\bf{Step 3.}}
{\it Condition  (\ref{C1S}) of Definition~\ref{C1} holds.}

 We use induction on the
number of nodes. The base is the nonsingular case, i.e.  the classical Torelli theorem. From now on we assume $X$  and $X'$ singular.

As usual, we denote the normalizations of $X$ and $X'$ both by $\Xn$.

 Let $S\in \Set X$ and $S'\in \Set X'$ be a pair of corresponding C1-sets, under the bijection
 described in the first step; set $h:=\#S=\#S'$.
 Let $\nu_S:Y_S\to X$ and $\nu_{S'}:Y'_{S'}\to  X'$ be the partial normalizations at $S$ and $S'$,
 and call $g_S=g-h$   their   arithmetic genus.
Recall that
$Y_S $ and $ Y'_{S'}$ have $h$ connected components, each of which is free from separating nodes
and has positive arithmetic genus.
We claim  that $Y_S$ and $Y'_{S'}$ are C1-equivalent.

Recall (see (\ref{stratumcl})) that we denote by $\PSb\subset \PXgb$ and $\ov{P_{S'}}\subset \PX'gb$ the closures
of all strata supported, respectively,  on $S$ and $S'$.
By what we said, the isomorphism $\phi_1$ induces an isomorphism
\begin{equation}
\label{PSS}
\PSb \cong \ov{P_{S'}}.
\end{equation}
By Lemma~\ref{stratum} we obtain that  $\PSb $
together with the restriction of the theta
divisor and the action of $J(Y_S)$ is naturally isomorphic to $  (J(Y_S)\curvearrowright \ov{P^{g_S-1}_{Y_S}}, \Theta(Y_S))$;
similarly for $\ov{P_{S'}}$. Therefore by (\ref{PSS})
we have
$$
 (J(Y_S)\curvearrowright \ov{P^{g_S-1}_{Y_S}}, \Theta(Y_S)) \cong  (J(Y'_{S'})\curvearrowright \ov{P^{g_{S}-1}_{Y'_{S'}}}, \Theta(Y'_{S'})).
$$
By Proposition \ref{stab}, the same holds if  $Y_S$ and $Y'_{S'}$ are replaced by their stabilizations, $\ov{Y_S}$
and $\ov{Y'_{S'}}$.
Therefore we can apply the induction hypothesis to $\ov{Y_S}$
and $\ov{Y'_{S'}}$
(which  are stable,  free from separating nodes, and have fewer nodes than $X$ and $X'$). We thus obtain   that
  $\ov{Y_S}$   is C1-equivalent , or T-equivalent, to    $\ov{Y'_{S'}}$.

 On the other hand the normalizations of $Y_S$ and $Y'_{S'}$
 are isomorphic, as they are equal to the normalizations of $X$ and $X'$.
 Furthermore, as   $\Gamma_X\equiv_{\rm cyc} \Gamma_{X'}$
 (by Step 2) the dual graphs of $Y_S$ and $Y'_{S'}$ are cyclically equivalent
 (by the same argument used   for Lemma~\ref{Tind}).
  Therefore,
 by Lemma~\ref{valeq},
 $Y_S$ is  T-equivalent, hence C1-equivalent,  to    ${Y'_{S'}}$. The claim is proved.

 Next, consider the normalization maps
 $$
  \nu:\Xn\stackrel{\mu}{\la}  Y_S \stackrel{\nu_S}{\la} X,\  \  \   \nu':\Xn\stackrel{\mu'}{\la} Y'_{S'} \stackrel{\nu_{S'}}{\la} X'
    $$
    where $\mu$ and $\mu'$ are the normalizations of $Y_S $ and $ Y'_{S'}$.
    As $Y_S $ and $ Y'_{S'}$ are C1-equivalent, the gluing sets
    $\mu^{-1}((Y_S)_{\text{sing}})$ and $\mu'^{-1}((Y'_{S'})_{\text{sing}})$
     are the same (cf. \ref{Tfin}).
The gluing sets of $\nu$ and $\nu '$ are obtained by adding to the above set the gluing sets of $S$ and $S'$.

By Proposition~\ref{theta},
we have that $\nu_S^{-1}(S)$ and $\nu_{S'}^{-1}(S')$ are uniquely determined
by the ppSSAV of $Y_S$ or of $Y_{S'}$, which are isomorphic.
Therefore
 $(\ov{P^{g_S-1}_{Y_S}}, \Theta(Y_S))$
uniquely determines $\mu^{-1}(\nu_S^{-1}(S))=\nu^{-1}(S)$ and
$\mu'^{-1}(\nu_{S'}^{-1}(S'))=\nu'^{-1}(S')$  on $\Xn$.
This is to say that, up to automorphisms of $\Xn$, the sets $\nu^{-1}(S)$ and $\nu'^{-1}(S')$
coincide.
Denote  $G_S:=\nu^{-1}(S)=\nu'^{-1}(S')$.
We also obtain that
 the  gluing set   of $\nu $ is equal to the gluing set   of   $\nu'$;
we call it $G_{\sing}=\nu^{-1}(\sing)=\nu'^{-1}(X'_{\text{sing}})$.
Of course $G_{\sing}$ is the disjoint union of all the gluing sets associated to  all the C1-sets of $X$.

Now we apply the previous argument to every remaining pair of corresponding C1-sets, as follows.
Pick a pair of corresponding C1-sets,  $U $ and $U'$, with $U\neq S$. Then, as before,
$Y_U$ and $Y'_{U'}$ are C1-equivalent, and their (same) ppSSAV uniquely determines
$$
G_U:=\nu^{-1}(U)=\nu'^{-1}(U')\subset G_{\sing}\smallsetminus G_S\subset \Xn.
$$
Therefore condition (\ref{C1S}) of
  Definition~\ref{C1} holds, i.e. $X$ and $X'$ are C1-equivalent. The proof is complete.
$\qed$

We used the following basic fact.

\begin{lemma}
\label{valeq}
Let $X$ and $X'$ be free from separating nodes; suppose that
their stabilizations are T-equivalent and  that   $\Gamma_{X}\equiv_{\rm cyc} \Gamma_{X'}$.
Then $X$ and $X'$ are T-equivalent.
\end{lemma}
\begin{proof}
Let $\ov{X}$ and $\ov{X'}$ be the  stabilizations of $X$ and $X'$.
Observe that the dual graph of $\ov{X}$ is  obtained from
$\Gamma_X$ by removing some vertices
of valence $2$ (corresponding to the exceptional components of $X$)
so that the two edges adjacent to every such vertex become a unique edge. Therefore
there is a natural isomorphism $H_1(\Gamma_X)\cong H_1(\Gamma_{\ov{X}})$. Moreover, this isomorphism fits in a commutative diagram
$$\xymatrix{
 H_1(\Gamma_X,\Z)
\ar[r]^{\eta_X}\ar[d]_{\cong}^{}
& \Div^{{\mo}}X^{\nu} \\
 H_1(\Gamma_{\ov{X}},\Z) \ar[r]^{\eta_{\ov{X}}}  &
\Div^{{\mo}}{\ov{X}}^{\nu} \ar@{^{(}->}[u] \\
}$$
where the right vertical arrow is induced by the obvious injection $\ov{X}^{\nu} \ha \Xn$.
The diagram immediately
yields that the map $\eta_{\ov{X}}$ is determined by $\eta_{{X}}$. The converse is also true,
in fact if $E\subset X$ is an exceptional component, and
$\pi_E:\Div^{{\mo}}X^{\nu}\to \Div^0E=\Div^0\pr{1}$ the projection
($\Div^0E$ is a factor of $\Div^{{\mo}}X^{\nu}$),
then the map $\pi_E\circ \eta_X$ is
uniquely determined up to an automorphism of $X$.
The same observation applies to $X'$, of course.

Now to prove the lemma, notice that  $X$ and $X'$  have  the same number of   irreducible components, because their dual graphs are cyclically equivalent.
 Denote by
 $\Xn_+\subset \Xn$, respectively by   $X'^{\nu}_+\subset X'^{\nu}$, the union of all components of $\Xn$,
 respectively of  $X'^{\nu}$, having positive genus.
To show that $\Xn \cong \Xnn$  it suffices to show that
$\Xn_+\cong X'^{\nu}_+$. This follows immediately from the fact that the normalizations of
$\ov{X}$ and $\ov{X'}$ are isomorphic.

Finally, by the initial observation, the maps $\etab$ and ${\eta_{X'}}$ are determined
by those of $\ov{X}$ and $\ov{X'}$, and hence Property
(\ref{Tdef3}) of Definition~\ref{Tdef} holds for $X$ and $X'$, because it holds for their stabilizations.
\end{proof}

\subsection{Torelli theorem: proof of the sufficient condition}
\label{pfsuf}
By Corollary~\ref{Jacobian-irr} and Remark~\ref{stab} it suffices to prove  the first part of Theorem~\ref{main},
i.e. we can assume that $X$ and $X'$ are  C1-equivalent curves   free from separating nodes.
By Proposition~\ref{Teq}  C1-equivalence and T-equivalence coincide;
so we can use the second concept, which is now more convenient. Indeed
 the proof
consists in
  applying some well known (some quite deep) facts about
ppSSAV, on which our definition of T-equivalence is   based.

By \cite{AN}, and by
  \cite[Sec. 5.5]{alex} (where a short  description, ad-hoc for the present case, is given)
   $\tgb (X)$   is determined  by a set of  ``combinatorial data"
   (partly known also to Mumford and Namikawa, see \cite[Chap. 18]{Nam4} and
\cite[Chap. 9.D]{NamT}).
Let us recall them.
Denote by $J(\Xn )^t$ be the  dual abelian variety of   $J(\Xn )$.
Now let
$$
\lambda_X:J(\Xn ) \stackrel{\cong}{\la} J(\Xn )^t
$$
be the isomorphism associated to the class of the Theta divisor of $\Xn$.

Let $\mathcal{P}$ be the universal, or Poincar\'e, line bundle on $J(\Xn ) \times J(\Xn )^t$.
Recall that its set of $k$-rational points,  $\mathcal{P}(k)$, defines a biextension,
the so-called  Poincar\'e biextension, of $J(\Xn ) \times J(\Xn )^t$ by $k^*$;
 \cite[Sect. 2 p. 311]{Mum} or \cite{breen}.

Then $\tgb (X)$ is uniquely determined by the following data.

\begin{enumerate}
\item
\label{fg}
The free abelian group $H_1(\Gamma _X, \Z)$.
\item
\label{dc}
the Delaunay decomposition of the real vector space $H_1(\Gamma _X, \R)$
associated to the  lattice $H_1(\Gamma _X, \Z)$, with respect to the Euclidean scalar product.
\item
\label{cm}
The classifying morphism of the semiabelian variety $J(X)$, together with its dual.
In our present situation, this  is the datum of the group homomorphism
$
c_X:H_1(\Gamma _X, \Z)\la J(\Xn )
$
already described in \ref{map-delta}, together with its dual
$$
c_X^t:H_1(\Gamma _X, \Z) \stackrel{c_X}{\la} J(\Xn ) \stackrel{\lambda_X}{\la}J(\Xn )^t.
$$

\item
\label{PD}
The equivalence class of a  trivialization of the pull back to $H_1(\Gamma _X, \Z)\times H_1(\Gamma _X, \Z)$
of the inverse of the Poincar\'e bi-extension; i.e. the class of a map
$$
\tau_X:H_1(\Gamma_X,\Z)\times H_1(\Gamma_X,\Z)\to
(c_X^t\times c_X)^* \mathcal{P}^{-1}(k).
$$
This is determined  by composing
$$
\eta_X\times \eta_X: H_1(\Gamma_X, \Z)\times H_1(\Gamma_X, \Z) \to
\Div^{\un{0}}X^{\nu}\times \Div^{\un{0}}X^{\nu}
$$
with the Deligne symbol
(see  \cite[XVII]{SGA4} and \cite[Sec. 5.5]{alex}).

\end{enumerate}
Let us show that  such data are the same for our   T-equivalent curves $X$ and $X'$.

As the graphs $\Gamma_X$ and  $\Gamma_{X'}$ are cyclically equivalent,
there is an isomorphism $\epsilon_H:H_1(\Gamma_X,\Z)\stackrel{\cong}{\to} H_1(\Gamma_{X'},\Z)$.
Such an isomorphism
   induces an isomorphism
$\Del(\Gamma_X)\cong \Del(\Gamma_{X'})$ between the Delaunay decompositions of $X$ and
of $X'$ (see \cite[Prop. 3.2.3(i)]{CV}). Therefore the data (\ref{fg}) and (\ref{dc}) are the same for $X$ and $X'$.

Since
 $ X^{\nu}= X'^{\nu}$, we have  $J(X^{\nu}) = J(X'^{\nu})$ and the  principal polarizations,
 of course, coincide:
 $$\lambda_X=\lambda_{X'}:
J(X^{\nu}) \la  J(X^{\nu})^t.
$$
The classifying morphism has been described in \ref{map-delta}.
From \ref{Tdef}(\ref{Tdef3}), we get the commutativity of the following diagram
$$\xymatrix{
c_X: H_1(\Gamma_X,\Z)
\ar[r]^{\  \  \  \eta_X}\ar[d]_{\cong}^{\epsilon_H}
& \Div^{{\mo}}X^{\nu}\ar[d]_{\cong}^{\alpha} \ar@{->>}[r]&J(X^{\nu})
\ar[d]_{\cong}^{\ov{\alpha}}\\
c_{X'}:H_1(\Gamma_{X'},\Z) \ar[r]^{\  \  \  \eta_{X'}}  &
\Div^{{\mo}}{X}^{\nu}  \ar@{->>}[r]&J(X^{\nu})\\
}$$
where $\ov{\alpha}\in \Aut (J(X^{\nu}))$ is the automorphism induced by $\alpha$
(recall that $J(X^{\nu})=\Pic^{\mo}\Xn$).
It is clear that the automorphisms of $J(\Xn )$   have no effect on
the isomorphism class of the semiabelian variety corresponding to the classifying morphisms.
This shows that data (\ref{cm}) are also the same for $X$ and $X'$.

Let now ${\mathcal P}'(k)$ be the Poincar\'e bi-extension of  $X'$; see  (\ref{PD}).
By what we said so far, it is clear that
 $$
 (\epsilon_H\times \epsilon_H)^*
(c_{X'}^t\times c_{X'})^* \mathcal{P}'^{-1}(k)\cong
(c_{X}^t\times c_{X})^* \mathcal{P}^{-1}(k).
$$
Now, the class of the  map  $\tau_X$ (respectively $\tau_{X'}$) is constructed using the  Deligne symbol
which
is canonically defined on the pull back of $\mathcal{P} ^{-1}(k)$  (respectively of  $\mathcal{P}'^{-1}(k)$) to
$\Div^{\mo}\Xn \times \Div ^{\mo}\Xn$.
Therefore, using the above isomorphism
and  the commutative diagram of \ref{Tdef}(\ref{Tdef3}),
we get
$$\tau_X=(\epsilon_H\times \epsilon_H)^*
\tau_{X'}: H_1(\Gamma_X, \Z)\times H_1(\Gamma_X,\Z) \to
(c_{X}^t\times c_{X})^* \mathcal{P} ^{-1}(k).$$
Therefore the data of part (\ref{PD}) are also the same for $X$ and $X'$.
We thus proved that the data defining   $\tgb(X)$ and $\tgb(X')$ are the
same, hence we are done. $\qed$

\section{The fibers of the Torelli morphism}
\label{fibsec}
\subsection{Injectivity locus and fiber cardinality of the Torelli morphism}

\

Where, in $\Mgb$, is the compactified Torelli morphism $\tgb$ injective?
At this point it is clear (as was already known to Namikawa, see \cite[Thm. 9.30(iv)]{NamT})
 that this is the case for irreducible  curves;
the question is thus really interesting for reducible curves.
To give it a precise answer we introduce some terminology.
\begin{nota}
\label{3econn}
A
connected  graph is  {\it 3-edge connected} if it
remains connected after removing any  two of its edges.
We need the following  characterization
(Corollary 2.3.4 of \cite{CV}).
{\it A connected graph free from separating edges is 3-edge connected if and only if every C1-set  has cardinality one.}

Note also that given two cyclically equivalent connected graphs, one is 3-edge connected if and
only if the other one is. 
In graph theory,  the definition of a $3$-edge connected graph
is usually  given for graphs having at least two vertices. Here we  do not make this assumption, so for us a graph with one vertex is always 3-edge connected.
\end{nota}

We shall call ``Torelli-curves" those stable curves for which the Torelli map is injective;
see Definition~\ref{Tcurve} and Theorem~\ref{Tinj}.
 We first illustrate a simple case.
\begin{example}
\label{C1ex}
The following is the simplest
 example of   C1-equivalent stable curves.
Let $\Xn=\Xnn=C_1\coprod C_2$,
where the $C_i$ are smooth of genus $g_i\geq 1$. Let  $p_i,q_i \in C_i$ be distinct points; now define 
$$
X=\frac{C_1\coprod C_2}{(p_1=p_2, q_1=q_2)} \  \text{ and }\  \  X'=\frac{C_1\coprod C_2}{(p_1=q_2,q_1=p_2)}.
$$
It is clear that $X$ and $X'$ are C1-equivalent.

Observe now that they are not isomorphic, unless one of them, $C_1$ say,
has an automorphism switching $p_1$ with $q_1$.

Indeed, suppose that there exists $\alpha_1\in \Aut C_1$ such that $\alpha_1(p_1)=q_1$
and $\alpha_1(q_1)=p_1$.
Then the automorphism $\phi \in \Aut \Xn$
which restricts to $\alpha_1$ on $C_1$ and to the identity on $C_2$,
descends to an isomorphism between $X$ and $X'$, since
$\nu'\circ \phi (p_1)=\nu '\circ \phi (p_2)$ and $\nu'\circ \phi (q_1)=\nu '\circ \phi (q_2)$.
This example, when $\alpha_1$ as above exists, is a special case of  Torelli curve,
defined as follows.
 \end{example}
\begin{defi}
\label{Tcurve}
A   stable curve $X$ such that $\sep=\emptyset$ is called a  {\it Torelli curve} if
for every C1-set $S$ such that $\#S=h\geq 2$,
  conditions (\ref{Ta}) and   (\ref{Tb}) below hold.
\begin{enumerate}
\item\label{Ta}
For every $i=1,\ldots, h-1$ there exists an automorphism $\alpha_i\in \Aut (Y_i)$
such that  $\alpha_i(p_i)=q_i$ and $\alpha_i(q_i)=p_i$,
(where      $Y_1,\ldots, Y_h$  are the connected components
of $Y_S$
and   $p_i,q_i\in Y_i$ are the two gluing points).
\item
\label{Tb}
There is an isomorphism as marked curves $ (Y_i;p_i,q_i)\cong (Y_j; p_j, q_j)$ for every  $i,j\leq h -1$; or  else $h =3$ and
there exists  $ \alpha_h\in \Aut (Y_h)$
such that  $\alpha_h(p_h)=q_h$ and $\alpha_h(q_h)=p_h$.
\end{enumerate}
\end{defi}
\begin{example}
  If  $\Gamma_X$ is 3-edge connected
$X$ is a Torelli curve, by \ref{3econn}. \end{example}

\begin{thm}
\label{Tinj}
Let $X$ be a stable curve free from separating nodes. Then
\begin{enumerate}
\item
\label{Tinj1}
$$\#\tgb^{-1}(\tgb (X))
\leq \left\lceil \frac{(g-2)!}{2} \right\rceil.
$$
Furthermore the bound is sharp, and can be obtained  with
$X$ a cycle curve
equal to the union of $g-1$ elliptic curves, no two of them    isomorphic.
\item
\label{Tinj2}
$\tgb^{-1}(\tgb (X))=\{X\}$ if and only if $X$ is a Torelli-curve.
\end{enumerate}
\end{thm}
\begin{proof}
By Theorem~\ref{main} the set $\tgb^{-1}(\tgb (X))$ is the C1-equivalence class of $X$.
The bound on the cardinality of the  C1-equivalence class follows from
Lemma~\ref{countC1}.

Now,  let $X$ be the union of $g-1$ smooth curves $C_1,\ldots, C_{g-1}$
of genus $1$, so that the dual graph of $X$ is a cycle of length $g-1$.
Suppose that $C_i\not\cong C_j$ for all $i\neq j$.
 The curve $X$ has a unique C1-set, namely $S=\sing$, and each curve $C_i$ contains
exactly two points of $G_S$, which we call $p_i$ and $q_i$.
With the notation of \ref{gldata}, let $[(\sigma_S, \psi_S)]$
be the gluing data of $X$. Since each $C_i$ has an automorphism exchanging
$p_i$ with $q_i$, varying the marking $\psi_S$ does not change the isomorphism
class of the curve $X$. On the other hand, any change in $\sigma_S$
(with the exception of $\sigma_S^{-1}$
of course)  changes the isomorphism class
of the curve, because no two $C_i$ are isomorphic. Therefore, we conclude
that the number of non-isomorphic curves that are C1-equivalent to $X$
is equal to $1$ if $g\leq 3$, and $(g-2)!/2$ if $g\geq 4$.
Part (\ref{Tinj1}) is proved.

For part (\ref{Tinj2})
it suffices to  prove the following.
Let $X$ be connected with  $\sep=\emptyset$;
$X$ is a Torelli curve  if and only if the only curve C1-equivalent to $X$ is $X$ itself.

Assume first that $X$ is a Torelli curve.
If $\Gamma_X$ is $3$-edge connected, then every C1-set has cardinality 1 by \ref{3econn},
therefore we conclude by Lemma~\ref{number-equiv}. We can henceforth assume
that $\Gamma_X$ is not $3$-edge connected.

Let $S\in \Set X$ have cardinality $h\geq 2$ (it exists by \ref{3econn}).
We claim that $\Aut  X$ acts transitively on the gluing data of $S$,
described in \ref{gldata}.
We use the notation of Definition~\ref{Tcurve}.
If $h =3$ and   $Y_i$ has an automorphism exchanging $p_i$ with $q_i$
for $i=1,2,3$, then the claim trivially holds.

Next, assume that the first $h-1$ marked components $(Y_i; p_i,q_i)$ are isomorphic and have
an automorphism switching the gluing points $p_i,q_i$.
Using  the set-up of \ref{gldata},
the gluing data are given by   an ordering of the components,
which we can assume has $Y_h$ as last element,
and by a marking of each pair  $(p_i, q_i)$ for all $i=1,\ldots, \gamma-1$.
Now $\Aut X$  acts transitively on the
orderings of the components, by permuting $Y_1,\ldots, Y_{h -1}$, which are all isomorphic
by isomorphisms preserving the gluing points.
Moreover for  $i=1,\ldots, \gamma-1$ each pair of points $(p_i, q_i)$ is permuted by the automorphism $\alpha_i$.
 The claim is proved.
Of course, the claim implies that $X$ is unique in its C1-equivalence class..

Conversely, let $X$ be the unique curve in its C1-equivalence class.
If every C1-set of $X$ has cardinality 1 then $\Gamma_X$ is 3-edge  connected (by \ref{3econn}) and we are done.

 So, let $S\in \Set X$ be such that $\#S\geq 2$ and let us check  that the conditions of Definition~\ref{Tcurve} hold.
With no loss of generality, and using the same notation as before,
we may order the connected components of $Y_S$
  so that  $q_i$ is glued to $p_{i+1}$ and $p_i$ is glued to $q_{i-1}$
(with the cyclic convention, so that $p_1$ is glued to $q_h$).
Assume that $Y_h$ has no automorphism exchanging  $p_h$ with $q_h$; let us change the gluing data of $X$ by switching $p_h$ with $q_h$, and by leaving everything else unchanged.
Then the corresponding curve is C1-equivalent to $X$, and hence it
is isomorphic to $X$, by hypothesis. Therefore,
the curve $W=\ov{X\smallsetminus Y_h}$
must admit an automorphism switching $p_1$ with $q_{h -1}$
(the two points glued to $q_h$ and $p_{h}$).
Now    it is easy to see, by induction on the number of
components of $W$, that such an automorphism exists if and only if $W$ is a union $h-1$ of marked components,
$(Y_i; p_i,q_i)$, all  isomorphic to  $(Y_1; p_1,q_1)$,
and if $Y_1$ has an involution switching $p_1, q_1$.
Therefore  $X$ is a Torelli curve.

If instead $Y_i$ has an automorphism exchanging the two gluing points
for every $i=1,\ldots,h$,
and no $h-1$ among the $Y_i$ are isomorphic, it is clear that for $h\geq 4$
there exist different orderings of the $Y_i$ giving different C1-equivalent curves.
Therefore we must have $h=3$, hence $X$ is a Torelli curve.
\end{proof}
The proof of the Theorem used the following lemmas.

\begin{lemma}\label{number-equiv}
Let $X$ be a connected nodal curve free from separating nodes.
Then the cardinality of the
C1-equivalence class of $X$ is at most
$$\prod_{S\in \Set X}2^{\#S -1}(\# S-1)! \  .$$
\end{lemma}
\begin{proof}
By the discussion in \ref{gldata},
the number of curves that are C1-equivalent to   $X$ is bounded above
by the product of the number of all gluing data for each $C1$-set
$X$.
The C1-sets with $\#S=1$ admit only one gluing data,  so they do not contribute.

Let $S$ be a C1-set of cardinality at least $2$.
Clearly there are    $2^{\#S}$ possible markings $\psi_S$,
and $(\#S-1)!$ possible choices for the
cyclic permutation $\sigma_S$.
Furthermore, recall that each gluing data can be given by two such pairs
$(\psi_S,\sigma_S)$,
namely  the two conjugate pairs  under the involution (\ref{invgd}).
This gives us a total of $2^{\#S -1}(\# S-1)!$ gluing data.
\end{proof}
We shall repeatedly use the following elementary
\begin{remark}\label{auto-elliptic}
Let $E$ be a connected nodal curve of genus at most $1$, free from separating nodes.
For any two smooth points $p, q$ of $E$, there exists an automorphism
 of $E$ exchanging $p$ and $q$.
\end{remark}

\begin{lemma}
\label{countC1}
Let $X$  be a connected curve of genus $g\geq 2$ free from separating nodes;
let $e$ be the number of its exceptional components.
Then the C1-equivalence class of $X$ has cardinality at most
$$
\left\lceil \frac{(g-2+e)!}{2}\right\rceil.
$$
 \end{lemma}
\begin{proof}
Throughout this proof, we denote by $\{Y\}_{C1}$ the C1-equivalence class of a nodal
curve $Y$. We will use induction on $g$.

We begin with the following claim.
Let $\ov{X}$ be the stabilization of $X$. If
$\Gamma_{\ov{X}}$ is $3$-edge connected, then $\#\{X\}_{C1}=1$.

Indeed there is a natural bijection between the C1-sets of $X$ and those of $\ov{X}$;
which we denote by $S\mapsto \ov{S}$.
By assumption, for every C1-set $\ov{S}$ of $\ov{X}$
the partial normalization
$\ov{Y}_{\ov{S}}$ of $\ov{X}$ at $\ov{S}$  is connected (since $\#\ov{S}=1$).
Now, for any $S\in \Set X$,
the partial normalization $Y_S$ of $X$ at $S$ is equal to the disjoint union
of
$\ov{Y}_{\ov{S}}$   together
with some copies of $\pr{1}$. Using this explicit description and
\ref{auto-elliptic}  we find that all the possible gluing data $[(\sigma_S, \psi_S)]$
of $S$ (see \ref{gldata}) give isomorphic curves, i.e. $X$ is unique inside
its C1-equivalence class. The claim is proved.

Now we start the induction argument. Let us treat the cases $g=2, 3$.

Using the above claim, it is easy to see that   to prove the Lemma for $g=2, 3$
we need only   worry about curves $X$ of genus $3$, whose stabilization $\ov{X}$ is the union
of two components $C_1$ and $C_2$ of genus $1$, meeting at two points.
If $e=0$ then $X$ is unique in its C1-equivalence
class by using  \ref{auto-elliptic}.
If $e>0$ then the  curves C1-equivalent to $X$ are obtained by inserting
two chains of exceptional components   between $C_1$ and $C_2$, one of length $e_1$ for every
$0\leq e_1\leq \lfloor e/2\rfloor $,
 and the other of length $e-e_1$. It is obvious that for different values of $e_1$ we get non isomorphic curves, and that we get all of the curves C1-equivalent to $X$ in this way.
 Therefore
$$\#\{X\}_{C1}=1+\left\lfloor \frac{e}{2}\right\rfloor\leq \left\lceil \frac{(e+1)!}{2}\right\rceil.$$

Assume now $g\geq 4$ and let $S\in\Set X$ such that $\#S=h$.
As usual, we write $Y_S=\coprod_1^hY_i$, with $Y_i$ free from separating nodes and of genus
$g_i:=g_{Y_i}$.
We order the connected components $Y_i$ of $Y_S$ in such a way that:
\begin{itemize}
 \item $Y_1,\ldots, Y_f$ have genus at least $4$;
\item $Y_{f+1},\ldots,Y_{f+k_3}$ have genus $3$;
\item $Y_{f+k_3+1},\ldots, Y_{f+k_3+k_2}$ have genus $2$;
\item $Y_{f+k_3+k_2+1}, \ldots, Y_{f+k_3+k_2+k_1}$ have genus $1$;
\item $Y_{f+k_3+k_2+k_1+1},\ldots, Y_h$ have genus $0$ and therefore are isomorphic to $\pr{1}$.
\end{itemize}

Let $e_i$ be the number of exceptional components of $X$ contained in $Y_i$;
then $Y_i$ has at most $e_i+2$ exceptional components.
We have the obvious relations
\begin{gather}
e=\sum_{i=1}^h e_i=\sum_{g_i\geq 2} e_i + \sum_{g_i=1} e_i+ h-f-k_3-k_2-k_1 \tag{*} \\
g-1=\sum_i g_i=\sum_{g_i\geq 2} g_i +  k_1. \tag{**}
\end{gather}

Consider now the gluing data $[(\sigma_S, \psi_S)]$ associated to $S$ (notation
as in \ref{gldata}).
Call, as usual, $\{p_i, q_i\}$ the two points of $G_S$ contained in the component $Y_i$.
Since all the components $Y_i$ with $g_i\leq 1$ have an automorphism that exchanges
$p_i$ and $q_i$ (by \ref{auto-elliptic}), if we compose the marking
$\psi_S$ with the involution of $G_h$ that exchanges
$s_i$ with $t_i$ (for all indices $i$ such that $g_i\leq 1$) the resulting curve
will be isomorphic to the starting one. Therefore, the number of possible non-isomorphic
gluing  data associated to $S$ is bounded above by $(h-1)!2^{f+k_2+k_3-1}$; 
since $g\geq 4$   this number is an integer (if $f=k_2=k_3=0$ then $h\geq 3$).
We conclude that
$$
\#\{ X\}_{{C1}}\leq  (h-1)!2^{f+k_3+k_2-1}\prod_{i=1}^h\#\{ Y_i\}_{C1}.
 $$

The components $Y_i$ of genus at most $1$ are  unique inside their C1-equivalence
class. For the components $Y_i$ of genus $g_i\geq 2$ we can apply the induction hypothesis
(note that $2\leq g_i<g$) and we get that
$$\#\{Y_i\}_{{C1}}\leq  \left\lceil \frac{(g_i-2+e_i+2)!}{2}\right\rceil=
\frac{(g_i+e_i)!}{2}
$$
By substituting into the previous formula, we get$$
\#\{ X\}_{{C1}}\leq  (h-1)!2^{f+k_3+k_2-1}\prod_{i=1}^{f+k_3+k_2} \frac{(g_i+e_i)!}{2}=
\frac{(h-1)!\prod_{g_i\geq 2} (g_i+e_i)!}{2}.  $$
The number of (non-trivial)  factors of the product  
$(h-1)!\prod_{g_i\geq 2} (g_i+e_i)!$
is equal to $ h-2+\sum_{g_i\geq 2} (g_i+e_i-1)$.  Using the formulas (*) and (**), we get
that
$$h-2+\sum_{g_i\geq 2} (g_i+e_i-1)=g-3+e-\sum_{g_i=1}e_i\leq g-3+e.
$$
Since the factorial $(g-2+e)!$ has a number of  factors equal to $g-3+e$, we conclude
from the above inequalities that
$$\#\{ X\}_{{C1}}\leq \frac{(h-1)!\prod_{g_i\geq 2} (g_i+e_i)!}{2}\leq \frac{(g-2+e)!}{2},
$$
as claimed.
\end{proof}
\begin{cor}
\label{5}
$\tgb^{-1}(\tgb (X))=
\{X\}$ for every $X\in \Mgb$ with $\sep=\emptyset$ if and only if $g\leq 4$.
\end{cor}

\begin{remark}
Consider a Torelli curve $X$ of genus at least $5$ with dual graph non 3-edge connected.
It is not hard to see that $X$ is the specialization of   curves
for which the Torelli morphism is not injective.
On the other hand we just proved that
   $\tgb^{-1}(\tgb (X))=\{X\}$.
Therefore the Torelli morphism, albeit injective at $X$, necessarily  ramifies at $X$.
\end{remark}

\subsection{Dimension of the fibers}

Let $X$ be a stable curve of genus $g$; now we shall assume that $\sep$ is not empty
and bound the dimension of the fiber of the Torelli map over $X$.

Recall the notation of (\ref{notsep}); the normalization of $X$ at $\sep$ is denoted $\w{X}$.
We denote by
$\w{\gamma}_0$ the number of connected components of $\w{X}$ of arithmetic
genus $0$, by $\w{\gamma}_1$ the number of those
of arithmetic genus $1$, and by $\w{\gamma_+}$ the number of those having positive  arithmetic genus, that is:
$$
\w{\gamma}_j:=\#\{i: \tilde{g}_i=j\}, \  j=0,1,\  \ \  \text{ and }\    \  \w{\gamma}_+:=\#\{i: \tilde{g}_i\geq 1\}.
$$

\begin{prop}\label{dim-fibers}
Let $X$ be a stable curve of genus $g\geq 2$. Then
$$\dim \tgb^{-1}(\tgb(X))=2 \w{\gamma_+} -\w{\gamma}_1-2
$$
(i.e. the maximal dimension of an irreducible  component of $\tgb^{-1}(\tgb(X))$ is equal  to  $2 \w{\gamma_+} -\w{\gamma}_1-2$).
\end{prop}
\begin{proof}
According to Theorem \ref{main},
  $\tgb(X)$
depends on (and determines) the C1-equivalence class of the stabilizations $\ov{\w{X_i}}$
of the components of $\w{X}$ such that $\w{g_i}>0$.
The C1-equivalence class of $\ov{\w{X_i}}$ determines
$\ov{\w{X_i}}$ up to a finite choice.
In particular, note that $\w{\gamma_0}$ and the number, call it $e$, of exceptional components
of  $\coprod_{\w{g_i}>0} \w{X_i}$  is not determined by $\tgb(X)$.

The dimension of the locus of curves in the fiber $\tgb^{-1}(\tgb(X))$ having the same topological
type of $X$ is equal to
\begin{equation}\label{dim-type}
2\#\sep  -3\w{\gamma_0}-\w{\gamma_1}-e.
\end{equation}
 Indeed, each separating node gives two parameters of freedom,
because we can  arbitrarily  choose the two branches of the node. The components $\w{X_i}$ of arithmetic genus
$0$ reduce the parameters by $3$ because they have a $3$-dimensional automorphism group, similarly
the components of arithmetic genus $1$ reduce the parameters by $1$.
Finally, each exceptional component of $\coprod_{\w{g_i}>0} \w{X_i}$ reduces the parameters
by $1$, because it contains at least one branch of one of the separating nodes and exactly
two branches of non-separating nodes.

Formula (\ref{dim-type}) shows that the curves $X'$ in the fiber $\tgb^{-1}(\tgb(X))$ whose topological
type attains the maximal dimension are the ones for which $e'=0$ (i.e. each positive genus component
of $\w{X'}$ is stable) and $\w{\gamma_0'}=0$ (i.e.
  $\w{X'}$ has no genus $0$ component).

  In particular, since  $\w{\gamma'_+}=\w{\gamma'}$, such a curve $X'$ has
  $\#X'_{\text{sep}}=\w{\gamma'_+} -1$
  separating nodes.
Applying   formula (\ref{dim-type}) to the curve $X'$ we obtain
$
\dim \tgb^{-1}(\tgb(X))\leq 2 \w{\gamma_+'} -\w{\gamma'}_1-2.
$
To conclude that equality holds we must check that the locus of curves $X'$ is not empty.
This is easy: given $\coprod_{\w{g_i}>0} \w{X_i}$
we can glue (in several ways) the   stabilizations of the  $\w{X_i}$  so that they form a tree.
This, by our results, yields curves in $\tgb^{-1}(\tgb(X))$.
\end{proof}

\begin{cor}\label{ineq-dim-fibers}
Let $X$ be a stable curve of genus $g$.
Then
$$\begin{sis}
& \dim \tgb^{-1}(\tgb(X)) \leq g-2 & \text{ with equality iff } \  \w{g_i}\leq 2 \text{ for all } i,\\
& \dim \tgb^{-1}(\tgb(X)) \geq \w{\gamma_+}-2 &
\text{ with equality iff } \  \w{g_i}\leq 1 \text{ for all } i.\\
\end{sis}$$
\end{cor}
\begin{proof}
The first inequality follows from the Proposition and
$$g=\w{\gamma_1}+\sum_{\w{g_i}\geq 2} \w{g_i}\geq \w{\gamma_1}+2 (\w{\gamma_+}-\w{\gamma_1})=
2\w{\gamma_+}-\w{\gamma_1},$$
with   equality if and only if all $\w{g_i}\leq 2$ for all $i$.

The second inequality follows from
$\w{\gamma_1}\leq \w{\gamma_+},$
with   equality if and only $\w{g_i}\leq 1$.
\end{proof}

Using Theorem~\ref{Tinj}  and Corollary~\ref{ineq-dim-fibers}, one
obtains that for $g\geq 3$ the locus in $\Mgb$ where $\tgb$ has finite fibers
is exactly the open subset of stable curves free from separating nodes; see
\cite[Thm. 9.30(vi)]{NamT} and \cite[Thm. 1.1]{vv} for the analogous results for the map $\tgvor$.
On the other hand  $\tgb$ is an isomorphism for $g=2$; again
see \cite[Thm. 9.30(v)]{NamT}.


\begin{thebibliography}{EGKH02}



\bibitem[Ale02]{alex1}
Alexeev, V.: {\it Complete moduli in the presence of semiabelian group
action. } Ann. of Math.
155 (2002), 611--708.

\bibitem[Ale04]{alex} Alexeev, V.: {\it Compactified Jacobians and Torelli
map.}
Publ. RIMS, Kyoto Univ. 40 (2004), 1241--1265.

\bibitem[AN99]{AN} Alexeev, V.; Nakamura, I.: {\it On Mumford's
construction of
degenerating abelian varieties.}  Tohoku Math. J. 51  (1999),  no. 3,
399--420.

\bibitem[ACGH]{ACGH} 
Arbarello, E.; Cornalba, M.; Griffiths, P. A.; Harris, J.: 
{\it Geometry of algebraic curves. Vol. I.} Grundlehren der Mathematischen
Wissenschaften 267.
Springer-Verlag, New York, 1985

\bibitem[AMRT]{AMRT} Ash, A.; Mumford, D.; Rapoport, M.; Tai, Y. :
{\it Smooth compactification of locally symmetric varieties.}
Lie Groups: History, Frontiers and Applications, Vol. IV. Math. Sci.
Press,
Brookline, Mass., 1975.

 \bibitem[B77]{beau} Beauville, A.: {\it Prym varietes and the Schottky
 problem.} Invent. Math.  41
 (1977), no. 2, 149-196.


\bibitem[Bre]{breen} Breen, L.: {\it Fonctions th\^eta et th\'eor\`eme du cube.}
Lecture Notes in Mathematics 980, Springer-Verlag, Berlin 1983.

\bibitem[Bri07]{brion} Brion, M.: {\it Compactification de l'espace des
modules des vari\'et\'es
ab\'eliennes principalement polaris\'ees (d'apr\`es V. Alexeev).}
S\'eminaire Bourbaki, Vol. 2005/2006,  Ast\'erisque  No. 311  (2007).

\bibitem[Cap94]{caporaso}
 Caporaso,  L.: {\it A compactification of the universal
Picard variety over the moduli space of stable curves.}
Journ. of the Amer. Math. Soc. 7 (1994), 589--660.


\bibitem[Cap07]{ctheta} Caporaso L.:
{\it Geometry of the theta divisor of a compactified jacobian.}
To appear in Journ. of the Europ. Math. Soc. 
 Vol. 11,  2009 pp. 1385-1427.
Available at arXiv:0707.4602.

\bibitem[CV09]{CV} Caporaso, L.; Viviani, F.: {\it Torelli theorem
for graphs and tropical curves.} Duke Math. Journ.  Vol. 153. No1, 2010 pp. 129-171.  Available at arXiv:0901.1389.

\bibitem[E97]{est}
Esteves, E.: {\it Very ampleness for theta on the compactified Jacobian.}  Math. Z.  226  (1997),  no. 2,
181-191. 

\bibitem[FC90]{FC} Faltings, G.;  Chai, C. L.:
{\it Degeneration of abelian varieties. With an appendix by David
Mumford.}
Ergebnisse der Mathematik und ihrer Grenzgebiete (3) 22,
Springer-Verlag, Berlin, 1990.






\bibitem[Mum68]{Mum} Mumford, D.: {\it Biextension of formal groups},
in Proceedings of the Bombay Colloquium on Algebraic Geometry, Tata Institute of
Fundamental Research Studies in Mathematics 4, London, Oxford University Press, 1968.




\bibitem[Nam73]{Namtor}
Namikawa, Y.: {\it On the canonical holomorphic map from the moduli space of stable curves to the Igusa monoidal transform.}  Nagoya Math. J.  Vol. 52 (1973), 197-259.

\bibitem[Nam76a]{Nam1} Namikawa, Y.: {\it A new compactification of the
Siegel space
and degeneration of Abelian varieties I.}  Math. Ann.  221  (1976), no. 2,
97--141.

\bibitem[Nam76b]{Nam2} Namikawa, Y.: {\it A new compactification of the
Siegel space
and degeneration of Abelian varieties II.}  Math. Ann.  221  (1976), no.
3, 201--241.


\bibitem[Nam79]{Nam4} Namikawa, Y.: {\it Toroidal degeneration of abelian
varieties II.}
Math. Ann.  245  (1979), no. 2, 117--150.

\bibitem[Nam80]{NamT} Namikawa, Y.: {\it Toroidal compactification of
Siegel spaces.}
Lecture Notes in Mathematics 812, Springer, Berlin, 1980.

\bibitem[OS79]{OS}
Oda, T.;  Seshadri,  C.S.: {\it Compactifications of the generalized
Jacobian variety.}
Trans. A.M.S. 253 (1979) 1-90.






\bibitem[SGA]{SGA4}
{\it Seminaire de g\'eom\'etrie alg\'ebrique du Bois-Marie 1963-1964.
Th\'eorie des topos et cohomologie \'etale des schemas}.
Tome 3. Exp. IX   XIX.
LNM 305.  Springer  (1973).


\bibitem[Sim94]{simpson}
 Simpson,  C. T.: {\it Moduli of representations of the fundamental group of a
smooth projective variety.} Inst. Hautes \'Etudes Sci. Publ. Math. 80 (1994), 5--79.

\bibitem[T13]{torelli}
Torelli,  R.: {\it Sulle variet\`a di Jacobi.} Rendiconti della reale accademia dei Lincei
sr. 5, vol 22 (1913), 98--103.

\bibitem[V03]{vv}Vologodsky, V.: {\it The extended Torelli and Prym
maps.} University of Georgia PhD thesis (2003).

\bibitem[Whi33]{Whi}
Whitney, H.: {\it $2$-isomorphic graphs.}  Amer. Journ. Math. {\bf 55} (1933), 245--254.


\end{thebibliography}
\end{document}